\documentclass[11pt,reqno]{amsart}

\usepackage{amsmath,amssymb, amsthm}
\usepackage{mathrsfs}
\usepackage{cite}
\usepackage{enumerate}
\usepackage{enumitem}
\usepackage[mathlines]{lineno}
\usepackage[colorlinks=true, allcolors=blue]{hyperref}
\usepackage{pdfpages} 
\usepackage{tikz}
\usepackage{amsmath}
\usetikzlibrary{calc,decorations.pathreplacing,fit,shapes.misc}
\usetikzlibrary{calc,decorations.pathmorphing}
\pgfdeclarelayer{background}
\pgfdeclarelayer{foreground}
\pgfdeclarelayer{front}
\pgfsetlayers{background,main,foreground,front}

\definecolor{morandi1}{RGB}{255, 153, 153}
\definecolor{morandi2}{RGB}{102, 204, 255}  
\definecolor{morandi3}{RGB}{255, 204, 102}  
\definecolor{morandi4}{RGB}{102, 222, 153} 
\definecolor{morandi5}{RGB}{204, 153, 255} 

\usepackage{subcaption}
\pgfkeys{
  /qedge/.is family, /qedge,
  offset/.initial=0pt,
  line width/.initial=1pt,
  draw/.initial=,
  fill/.initial=,
  fill opacity/.initial=1,
}

\newcommand{\qedgeSeven}[8][]{
  \pgfkeys{/qedge, offset=0pt, line width=1pt, draw=, fill=, fill opacity=1, #1}
  \edef\qoffs{\pgfkeysvalueof{/qedge/offset}}
  \edef\qlw{\pgfkeysvalueof{/qedge/line width}}
  \edef\qdrawopt{\pgfkeysvalueof{/qedge/draw}}
  \edef\qfillopt{\pgfkeysvalueof{/qedge/fill}}
  \edef\qfillopacity{\pgfkeysvalueof{/qedge/fill opacity}}

  \ifdim\qoffs=0pt\relax
    \ifx\qfillopt\empty
      \fill[fill opacity=\qfillopacity] (#2)--(#3)--(#4)--(#5)--(#6)--(#7)--(#8)--cycle;
    \else
      \fill[\qfillopt, fill opacity=\qfillopacity] (#2)--(#3)--(#4)--(#5)--(#6)--(#7)--(#8)--cycle;
    \fi

    \ifx\qdrawopt\empty
      \draw[line width=\qlw] (#2)--(#3)--(#4)--(#5)--(#6)--(#7)--(#8)--cycle;
    \else
      \draw[line width=\qlw,\qdrawopt] (#2)--(#3)--(#4)--(#5)--(#6)--(#7)--(#8)--cycle;
    \fi
  \else

    \coordinate (1p) at ($#2!\qoffs!-90:#8$);
    \coordinate (1n) at ($#2!\qoffs! 90:#3$);

    \coordinate (2p) at ($#3!\qoffs!-90:#2$);
    \coordinate (2n) at ($#3!\qoffs! 90:#4$);

    \coordinate (3p) at ($#4!\qoffs!-90:#3$);
    \coordinate (3n) at ($#4!\qoffs! 90:#5$);

    \coordinate (4p) at ($#5!\qoffs!-90:#4$);
    \coordinate (4n) at ($#5!\qoffs! 90:#6$);

    \coordinate (5p) at ($#6!\qoffs!-90:#5$);
    \coordinate (5n) at ($#6!\qoffs! 90:#7$);

    \coordinate (6p) at ($#7!\qoffs!-90:#6$);
    \coordinate (6n) at ($#7!\qoffs! 90:#8$);

    \coordinate (7p) at ($#8!\qoffs!-90:#7$);
    \coordinate (7n) at ($#8!\qoffs! 90:#2$);

    \def\nqhedge{%
      (1p) let \p1=($(1p)-#2$), \p2=($(1n)-#2$) in
        arc[start angle={atan2(\y1,\x1)},
            delta angle={atan2(\y2,\x2)-atan2(\y1,\x1)-360*(atan2(\y2,\x2)-atan2(\y1,\x1)>0)},
            x radius=\qoffs, y radius=\qoffs] --
      (2p) let \p1=($(2p)-#3$), \p2=($(2n)-#3$) in
        arc[start angle={atan2(\y1,\x1)},
            delta angle={atan2(\y2,\x2)-atan2(\y1,\x1)-360*(atan2(\y2,\x2)-atan2(\y1,\x1)>0)},
            x radius=\qoffs, y radius=\qoffs] --
      (3p) let \p1=($(3p)-#4$), \p2=($(3n)-#4$) in
        arc[start angle={atan2(\y1,\x1)},
            delta angle={atan2(\y2,\x2)-atan2(\y1,\x1)-360*(atan2(\y2,\x2)-atan2(\y1,\x1)>0)},
            x radius=\qoffs, y radius=\qoffs] --
      (4p) let \p1=($(4p)-#5$), \p2=($(4n)-#5$) in
        arc[start angle={atan2(\y1,\x1)},
            delta angle={atan2(\y2,\x2)-atan2(\y1,\x1)-360*(atan2(\y2,\x2)-atan2(\y1,\x1)>0)},
            x radius=\qoffs, y radius=\qoffs] --
      (5p) let \p1=($(5p)-#6$), \p2=($(5n)-#6$) in
        arc[start angle={atan2(\y1,\x1)},
            delta angle={atan2(\y2,\x2)-atan2(\y1,\x1)-360*(atan2(\y2,\x2)-atan2(\y1,\x1)>0)},
            x radius=\qoffs, y radius=\qoffs] --
      (6p) let \p1=($(6p)-#7$), \p2=($(6n)-#7$) in
        arc[start angle={atan2(\y1,\x1)},
            delta angle={atan2(\y2,\x2)-atan2(\y1,\x1)-360*(atan2(\y2,\x2)-atan2(\y1,\x1)>0)},
            x radius=\qoffs, y radius=\qoffs] --
      (7p) let \p1=($(7p)-#8$), \p2=($(7n)-#8$) in
        arc[start angle={atan2(\y1,\x1)},
            delta angle={atan2(\y2,\x2)-atan2(\y1,\x1)-360*(atan2(\y2,\x2)-atan2(\y1,\x1)>0)},
            x radius=\qoffs, y radius=\qoffs] --
      cycle
    }

    \ifx\qfillopt\empty
      \fill[fill opacity=\qfillopacity] \nqhedge;
    \else
      \fill[\qfillopt, fill opacity=\qfillopacity] \nqhedge;
    \fi

    \ifx\qdrawopt\empty
      \draw[line width=\qlw, rounded corners=\qoffs] \nqhedge;
    \else
      \draw[line width=\qlw, \qdrawopt] \nqhedge;
    \fi
  \fi
}

\usepackage[margin=2.5cm]{geometry}

\newtheorem{theorem}{Theorem}[section]
\newtheorem{claim}[theorem]{Claim}
\newtheorem{lemma}[theorem]{Lemma}
\newtheorem{definition}[theorem]{Definition}
\newtheorem{problem}[theorem]{Problem}
\newtheorem{conjecture}[theorem]{Conjecture}
\newtheorem{fact}[theorem]{Fact}
\newtheorem{proposition}[theorem]{Proposition}

\newcommand{\cH}{{\mathcal H}}
\newcommand{\cG}{{\mathcal G}}

\newcommand{\cQ}{{\mathcal Q}}
\newcommand{\cP}{{\mathcal P}}
\newcommand{\cB}{{\mathcal B}}
\newcommand{\cR}{{\mathcal R}}

\newcommand{\cN}{{\mathcal N}}

\newcommand{\cL}{{\mathcal L}}
\newcommand{\eps}{{\varepsilon}}
\numberwithin{equation}{section}

\title{Exact minimum co-degree conditions for $\ell$-Hamiltonicity in hypergraphs}
\author{Luyining Gan}
\thanks{LG was partially supported by the National Natural Science Foundation of China (12401446).}
\address{School of Mathematical Sciences\\ Beijing University of Posts and Telecommunications\\ Beijing\\ China}
\email{elainegan@bupt.edu.cn}

\author{Jie Han}
\thanks{JH was partially supported by the National Natural Science Foundation of China (12371341).}
\address{School of Mathematics and Statistics\\ Beijing Institute of Technology\\ Beijing\\ China}
\email{han.jie@bit.edu.cn}

\author{Huan Xu}
\address{School of Mathematical Sciences\\ Beijing University of Posts and Telecommunications\\ Beijing\\ China}
\email{h.xu@bupt.edu.cn}
\date{}

\begin{document}

\begin{abstract}
Suppose $1\le \ell <k$ such that $(k-\ell)\nmid k$.
Given an $n$-vertex $k$-uniform hypergraph $\cH$, for all $k/2<\ell< 3k/4$ and sufficiently large $n\in (k-\ell)\mathbb N$, we prove that if $\cH$ has minimum co-degree at least $\frac{n}{\lceil \frac{k}{k-\ell}\rceil (k-\ell)}$, then $\cH$ contains a Hamilton $\ell$-cycle, which partially verifies a conjecture of Han and Zhao and (partially) resolves a problem of R\"odl and Ruci\'nski.
Moreover, we show that assuming minimum co-degree $\frac{n}{\lceil \frac{k}{k-\ell}\rceil (k-\ell)}+\frac{k^2}2$ is enough for all $\ell$. 
\end{abstract}

\maketitle

\section{Introduction}\label{int}
The Hamilton cycle problem is one of the central topics in graph theory. 
A classical result of Dirac~\cite{DIR1952} states that if an $n$-vertex graph $G$ with $n\ge 3$ has minimum degree $\delta(G)\ge n/2$, then it contains a Hamilton cycle, and the quantity $n/2$ is the best possible.
Over the past two to three decades, there has been a strong focus on extending Dirac's theorem to (uniform) hypergraphs.

Denote by $[r]$ the set of integers from $1$ to $r$.
For an integer $k\ge 0$, a set of $k$ elements is referred to as a \emph{$k$-set}.
For a set $A$, we use $\binom{A}{k}$ to denote the collection of all $k$-subsets of $A$.
Given $k\ge 2$, a \emph{$k$-uniform hypergraph} $\cH$ (for short, \emph{$k$-graph}) is a pair $(V,E)$, 
where $V:=V(\cH)$ is a vertex set and $E:=E(\cH)\subseteq \binom{V}{k}$ is a family of $k$-subsets of $V$. 
We write $e(\cH):=|E(\cH)|$ for the number of edges in $\cH$.
Given a $k$-graph $\cH$ and a vertex set $S\in \binom{V}{\ell}$ with $1\le \ell<k$, 
let $\deg_{\cH}(S)$ (the subscript $\cH$ is omitted if no confusion can arise) be the number of $(k-\ell)$-subsets $Y\in \binom{V}{k-\ell}$ 
such that $S\cup Y$ is an edge in $\cH$.
The \emph{minimum $\ell$-degree} $\delta_{\ell}(\cH)$ of $\cH$ is the minimum of $\deg_{\cH}(S)$ over all $\ell$-sets $S$ in $\cH$, that is, 
$\delta_{\ell}(\cH)=\min \{\deg_{\cH}(S):S\in \binom{V}{\ell}\}$. 
In particular, 
$\delta_{k-1}(\cH)$ is usually referred to as the minimum co-degree of $\cH$.
For a $k$-graph $\cH$ and a set $A \subseteq V$, $e_{\cH}(A)$ denotes the number of edges of $\cH$ contained in $A$.

There are several notions of hypergraph cycles, e.g. uniform cycles and Berge cycles.
In this paper, we consider uniform cycles defined as follows.
For $1\le \ell <k$, a $k$-graph is called a \emph{$k$-uniform $\ell$-cycle} if there exists a cyclic ordering of its vertices such that every edge consists of $k$ consecutive vertices and two consecutive edges share exactly $\ell$ vertices.
Note that a \emph{$k$-uniform $\ell$-path} is defined similarly except that the ordering is linear instead of cyclic.
For an $\ell$-path $\cP=v_1v_2\dots v_t$, the $\ell$-tuples $v_1\dots v_\ell$ and $v_tv_{t-1}\dots v_{t-\ell+1}$ are called the \emph{$\ell$-ends} (or \emph{ends}, for short) of $\cP$.
A $k$-graph on $n$ vertices contains a \emph{Hamilton $\ell$-cycle}
if it contains an $\ell$-cycle as a spanning subhypergraph.
Note that $(k-\ell)\mid n$ is a necessary condition for containing a Hamilton $\ell$-cycle.

In 1999, Katona and Kierstead~\cite{KK1999} proved that if an $n$-vertex $k$-graph $\cH$ satisfies $\delta_{k-1}(\cH)\ge (1-\frac{1}{2k})n+4-k-\frac{5}{2k}$ then $\cH$ contains a Hamilton $(k-1)$-cycle (also called \emph{tight Hamilton cycle}), 
and further conjectured that the best possible bound should be 
$\delta_{k-1}(\cH)\ge \frac{n-k+3}{2}$.

Major progress towards this conjecture was made by R\"odl, Ruci\'nski and Szemer\'edi, who~\cite{RRS2006, RRS2008} proved an asymptotic version of this conjecture in~\cite{KK1999} which asserted that for $k\ge 3, \gamma> 0$ and  sufficiently large $n$, every $k$-graph $\cH$ on $n$ vertices with $\delta_{k-1}(\cH)\ge (\frac{1}{2}+\gamma)n$ contains a tight Hamilton cycle.
For $k=3$, they~\cite{RRS11} further improved the minimum co-degree condition to a tight one, e.g.~$\delta_2(\cH)\ge n/2$.
Later Markstr\"om and Ruci\'nski~\cite{MR2011} showed that this indeed provides asymptotically tight minimum co-degree conditions for all $\ell\in [k-1]$ such that $(k-\ell)\mid k$.
%
For other $\ell$, that is, when $(k-\ell)\nmid k$, this threshold can be lowered significantly.
H\`an and Schacht~\cite{HS2010} proved that for $k\ge 3,\gamma>0$ and $1\le \ell<k/2$, every $n$-vertex $k$-graph $\cH$ with $\delta_{k-1}(\cH)\ge (\frac{1}{2(k-\ell)}+\gamma)n$ contains a Hamilton $\ell$-cycle.
This full threshold was established by K\"uhn,  Mycroft and Osthus~\cite{KMO10}, that is,
assuming minimum co-degree at least $\frac{n}{\lceil\tfrac{k}{k-\ell}\rceil(k-\ell)}+\gamma n$ contains a Hamilton $\ell$-cycle.
In the nice survey~\cite{VA10} of R\"odl and Ruci\'nski, they proposed the following problem, asking about the \emph{exact} minimum co-degree condition.

\begin{problem}[Problem 2.9, \cite{VA10}]
\label{prob}
Fix integer $k\geq 3$ and $1\leq \ell <k$. Assume that $n\in (k-\ell)\mathbb{N}$ is sufficiently large. Determine the exact value of minimum co-degree for Hamilton $\ell$-cycle in $k$-graph on $n$ vertices.
\end{problem}

So far, this has been resolved for $k=3$ and $\ell=2$ by R\"odl, Ruci\'nski and  Szemer\'edi~\cite{RRS11}, for $k=3$ and $\ell=1$ by Czygrinow and Molla~\cite{CM14}, and for $k\ge 3$ and $1\le \ell < k/2$ by Han and Zhao~\cite{HZ15}.
For the case $\ell=k/2$, Problem~\ref{prob} has been resolved for $k=4$ by Garbe and Mycroft~\cite{GM2018}, and for even $k\ge 6$ by H\`an, Han and Zhao~\cite{HHZ2022}.
For other results on Hamilton cycles, see~\cite{KO2014, DD2006, PDRD2011, BGHS1978, BMSSS2018, CHWY2025, HSW2025, LS2022, BHS2013, GPW12, HZ2015, BJ2017, HZ2016, RRRSS2019} and references therein.


Han and Zhao~\cite{HZ15} further conjectured the exact minimum co-degree condition for all $k\geq 3$ and $1\leq \ell < k$ with $(k-\ell)\nmid k$.
For brevity, define
\[s:=\left\lceil \frac{k}{k-\ell}\right\rceil.\]

\begin{conjecture}\emph{\cite{HZ15}}\label{conj:co-degree}
Fix integers $k\geq 3$ and $k/2< \ell < k$ such that $(k-\ell)\nmid k$. 
Assume that $n\in (k-\ell)\mathbb{N}$ is sufficiently large. If $\cH=(V,E)$ is an $n$-vertex $k$-graph such that $\delta_{k-1}(\cH)\geq \frac{n}{s(k-\ell)}$,
then $\cH$ contains a Hamilton $\ell$-cycle.
\end{conjecture}

\subsection{Results}

In this paper, we make the following progress towards Conjecture~\ref{conj:co-degree}.

\begin{theorem}\label{thm:exadeg}\emph{(Exact result for $\ell<3k/4$).}
    Let integers $k\geq 3$ and $k/2< \ell<3k/4$ such that $(k-\ell)\nmid k$. 
    Assume that $n\in (k-\ell)\mathbb{N}$ is sufficiently large. If $\cH=(V,E)$ is an $n$-vertex $k$-graph such that 
\begin{eqnarray}\label{eqn:conj value}
\begin{aligned}    
    \delta_{k-1}(\cH)\geq \frac{n}{s(k-\ell)},
\end{aligned}
\end{eqnarray}
then $\cH$ contains a Hamilton $\ell$-cycle.
\end{theorem}

\begin{theorem}\label{thm:asydeg}\emph{(Asymptotic result for all $\ell$).}
Let integers $k\geq 3$ and $k/2< \ell < k$ such that $(k-\ell)\nmid k$. 
Assume that $n\in (k-\ell)\mathbb{N}$ is sufficiently large. If $\cH=(V,E)$ is an $n$-vertex $k$-graph such that
\begin{eqnarray}\label{aa1}
\begin{aligned}
\delta_{k-1}(\cH)\geq \frac{n}{s(k-\ell)}+\frac{k^2}2,
\end{aligned}
\end{eqnarray}
then $\cH$ contains a Hamilton $\ell$-cycle.
\end{theorem}
Our proof follows the existing stability framework and splits the proof into a non-extremal case and an extremal case.
\begin{definition}
Let $\Delta > 0$, an $n$-vertex $k$-graph $\cH$ is called $\Delta$-extremal if there is a set $B \subseteq V(\cH)$, such that 
$|B| = \left\lfloor \left(1-\frac{1}{s(k - \ell)}\right) n \right\rfloor$
and $e(B) \leq \Delta n^k.$
\end{definition}

\begin{theorem}\emph{(Non-extremal case).}\label{thm:nonextremal}
For any integer $k \geq 3$, $1\leq \ell < k$ such that $(k-\ell)\nmid k$ and $0<\Delta<1$, there exists $\gamma > 0$ such that the following holds. 
Suppose that $\cH=(V,E)$ is an $n$-vertex $k$-graph such that $n \in (k - \ell)\mathbb{N}$ is sufficiently large. 
If $\cH$ is not $\Delta$-extremal and satisfies 
$\delta_{k-1}(\cH) \geq \left( \frac{1}{s(k - \ell)} - \gamma \right) n$,
then $\cH$ contains a Hamilton $\ell$-cycle.
\end{theorem}

\begin{theorem}\emph{(Extremal case).}\label{thm:extremal}
For any integer $k \geq 3$ such that $(k-\ell)\nmid k$, there exists $\Delta>0$ such that the following holds. 
Suppose that $\cH=(V,E)$ is an $n$-vertex $k$-graph such that $n \in (k - \ell)\mathbb{N}$ is sufficiently large.
If $\cH$ is $\Delta$-extremal with either 
\begin{itemize}
     \item $\delta_{k-1}(\cH) \geq \frac{n}{s(k - \ell)}$ for $k/2< \ell<3k/4$,
or    \item $\delta_{k-1}(\cH) \geq \frac{n}{s(k - \ell)}+\frac{k^2}{2}$ for $k/2< \ell<k$, 
\end{itemize}
then $\cH$ contains a Hamilton $\ell$-cycle.
\end{theorem}

\subsection{Hightlights of proof ideas}
Our proof of the non-extremal case (Theorem~\ref{thm:nonextremal}) follows the absorption framework developed in previous works and in fact most of the parts are already proven.
Indeed, an absorption lemma and a path-connecting lemma (which work under a much weaker minimum c-odegree condition) have been proven by K\"uhn, Mycroft and Osthus in~\cite{KMO10}, so that it suffices to prove a path cover lemma, which we derive from the (weak) regularity lemma and a tiling result in~\cite{GHZ19}.

For the extremal case (Theorem~\ref{thm:extremal}), as in previous works, we use the structural information to derive a partition of the vertex set of $\cH$ and analyze it.
There are then two steps for building a Hamilton $\ell$-cycle in $\cH$.
Indeed, we find a short path $\cP$ (of length $o(n)$) that covers all atypical vertices and leaves the rest part of the partition with a correct ratio; and then we extend $\cP$ to a Hamilton cycle.
New ideas are employed in both steps:
to build $\cP$, we use a (best-known) upper bound of the Tur\'an number of the hypergraph tight path by F\"uredi et al.~\cite{FJKMV20}; for the second step, we use a version of the Lov\'asz Local Lemma for random injections by Lu and Sz\'ekely~\cite{LS2007}.
Moreover, we show that the second step can be completed under the conjectured minimum co-degree condition, and thus reduce the validation of Conjecture~\ref{conj:co-degree} to the completion of the first step (see Theorem~\ref{thm:main result}).
To prove Theorem~\ref{thm:extremal}, we show that the first step can be achieved under the minimum co-degree assumptions therein.

For the rest of the paper, we first introduce several relevant lemmas and prove the non-extremal case in Section~\ref{sec:2}.
We then provide the proof of the extremal case in Section~\ref{sec:3}.

\section{Auxiliary lemmas and proof of Theorem~\ref{thm:nonextremal}}\label{sec:2}
In this section, we start with some relevant properties and then state the proof of Theorem~\ref{thm:nonextremal}.
\subsection{Auxiliary lemmas}
For $r \geq 1$ and $k \geq 2$, let
$P_r^k $ be the $k$-uniform $(k-1)$-path of length $r$. Denote by $\text{ex}(n, P_r^k)$ the maximum number of edges in an $n$-vertex $k$-graph not containing $P_r^k$. 
F\"uredi et al.~\cite{FJKMV20} gave the following upper bound of $\text{ex}(n, P_r^k)$.
\begin{lemma}\label{turan number}\emph{\cite{FJKMV20}}
For $n \geq 1$, $k \geq 2$ and $r \geq 1$,
\[
\mathrm{ex}(n, P^{k}_{r}) \leq 
\begin{cases} 
\frac{r-1}{2}\binom{n}{k-1}, & \text{if } k \text{ is even}, \\[10pt]
\frac{1}{2} \left( r + \left\lfloor \frac{r-1}{k} \right\rfloor \right) \binom{n}{k-1}, & \text{if } k \text{ is odd}.
\end{cases}
\]
\end{lemma}

Let $\cH=(V, E)$ be a $k$-graph and let $A_{1}, \ldots, A_{k}$ be mutually disjoint non-empty subsets of $V$. We define $e\left(A_{1}, \ldots, A_{k}\right)$ to be the number of \emph{crossing} edges, that is, those with exactly one vertex in each $A_{i}, i \in[k]$. The density of $\cH$ with respect to $\left(A_{1}, \ldots, A_{k}\right)$ is defined as
\[
d\left(A_{1}, \ldots, A_{k}\right)=\frac{e\left(A_{1}, \ldots, A_{k}\right)}{\left|A_{1}\right| \cdots\left|A_{k}\right|}.
\]

Given $\varepsilon>0$ and $d \geq 0$, we say a $k$-tuple $\left(V_{1}, \ldots, V_{k}\right)$ of mutually disjoint subsets $V_{1}, \ldots, V_{k} \subseteq V$ is $(\varepsilon, d)$-\textit{regular} if
\[
|d(A_{1}, \ldots, A_{k})-d| \leq \varepsilon
\]
for all $k$-tuples of subsets $A_{i} \subseteq V_{i}, i \in[k]$, satisfying $|A_{i}| \geq \varepsilon|V_{i}|$. We say $(V_{1}, \ldots, V_{k})$ is $\varepsilon$-\textit{regular} if it is $(\varepsilon, d)$-regular for some $d \geq 0$. It is immediate from the definition that in an $(\varepsilon, d)$-regular $k$-tuple $\left(V_{1}, \ldots, V_{k}\right)$, if $V_{i}^{\prime} \subseteq V_{i}$ has size $\left|V_{i}^{\prime}\right| \geq c\left|V_{i}\right|$ for some $c \geq \varepsilon$, then $\left(V_{1}^{\prime}, \ldots, V_{k}^{\prime}\right)$ is $(\varepsilon / c, d)$-regular.

Now we state the Weak Regularity Lemma for hypergraphs, which is a straightforward extension of Szemerédi's regularity lemma for graphs~\cite{Szm1978}.

\begin{theorem}[Weak Regularity Lemma]\label{thm:weak lem}
Given $t_{0} \geq 0$ and $\varepsilon>0$, there exist $T_{0}=T_{0}\left(t_{0}, \varepsilon\right)$ and $n_{0}=n_{0}\left(t_{0}, \varepsilon\right)$ such that, for every $k$-graph $\cH=(V, E)$ on $n>n_{0}$ vertices, there exists a partition $V=V_{0} \cup V_{1} \cup \cdots \cup V_{t}$ such that:
\begin{enumerate}
    \item[(i)] $t_{0} \leq t \leq T_{0}$,
    \item[(ii)] $|V_{1}|=|V_{2}|=\cdots=|V_{t}|$ and $|V_{0}| \leq \varepsilon n$,
    \item[(iii)] for all but at most $\varepsilon\binom{t}{k}$ $k$-subsets $\left\{i_{1}, \ldots, i_{k}\right\} \subset[t]$, the $k$-tuple $\left(V_{i_{1}}, \ldots, V_{i_{k}}\right)$ is $\varepsilon$-regular.
\end{enumerate}
\end{theorem}

The partition given in Theorem~\ref{thm:weak lem} is called an $\varepsilon$-\textit{regular partition} of $\cH$. Given an $\varepsilon$-regular partition of $\cH$ and $d\geq0$, we refer to $V_i, i \in [t]$ as \textit{clusters} and define the \textit{cluster hypergraph} $\cR = \cR(\varepsilon, d)$ with vertex set $[t]$ in which $\{i_1, \ldots, i_k\} \subseteq [t]$ is an edge if and only if $(V_{i_1}, \ldots, V_{i_k})$ is $\varepsilon$-regular and $d(V_{i_1}, \ldots, V_{i_k}) \ge d$.

The following proposition shows that the cluster hypergraph inherits the minimum co-degree of the original hypergraph. The proof is standard and very similar to that of \cite[Proposition 16]{HS2010}, so we omit the proof.

\begin{proposition}\emph{\cite{HS2010}} \label{cor:re-partition}
Given $c,\eps,d>0$, integers $k\ge 3$ and $t_0$, there exist $T_0$ and $n_0$ such that the following holds. 
Let $\cH$ be a $k$-graph on $n > n_0$ vertices with $\delta_{k-1}(\cH)\ge cn$.
Then $\cH$ has an $\varepsilon$-regular partition $V_0\cup V_1\cup \cdots \cup V_t$ with $t_0 \le t \le T_0$, and in the cluster hypergraph $R = R(\varepsilon, d)$, all but at most $\sqrt{\varepsilon} t^{k-1}$ $(k-1)$-subsets $S$ of $[t]$ satisfy $\deg_\cR(S) \ge (c-d-\sqrt{\varepsilon})t-(k-1)$.
\end{proposition}

The following lemma allows a small number of $(k-1)$-subsets of $V(\cH)$ to have low degree.
Given two $k$-graphs $\cG$ and $\cH$, a \emph{$\cG$-tiling} in $\cH$ is a subhypergraph of $\cH$ that consists of vertex-disjoint copies of $\cG$. 
Denote by $K^{(k)}(a_1,\dots,a_k)$ the complete $k$-partite $k$-graph with parts of size $a_1,\dots,a_k$.


\begin{lemma}\emph{\cite{GHZ19}}\label{lem:GHZ-2}
Fix integers $k \ge 2$ and $a < b$, $0 < \gamma_3 \ll 1/m$ and let $K:= K^{(k)}(a, b, \dots, b)$ and $m=a+(k-1)b$. For any $\beta > 0$ and $\Delta\ge 5bk^2\gamma_3$, there exist $\varepsilon > 0$ and an integer $n_0$ such that the following holds. Suppose $\cH$ is a $k$-graph on $n > n_0$ vertices with $\deg(S) \ge (\frac{a}{m}- \gamma_3)n$ for all but at most $\varepsilon n^{k-1}$ sets $S \in \binom{V(\cH)}{k-1}$, then $\cH$ has a $K$-tiling that covers all but at most $\beta n$ vertices or $\cH$ is $\Delta$-extremal.

\end{lemma}

We need the lemma following from the proof of~\cite[Lemma 7.3]{KMO10} by choosing the size of the vertex set of an $\ell$-path. Recall that $s:=\lceil \frac{k}{k-\ell}\rceil$.

\begin{lemma}\label{lem:color}\emph{\cite{KMO10}}
Let $\cP$ be an $\ell$-path of length $\lambda s$ for some $\lambda\in \mathbb N$.
Then there is a $k$-coloring of an $\ell$-path $\cP$ with colors $1, \dots, k$ such that color $k$ is used $\lambda$ times and the sizes of all other color classes are as equal as possible. 
\end{lemma}


\begin{proof}
Let an $\ell$-path $\cP:= x_1x_2\dots x_{\ell+\lambda s(k-\ell)}$.
Then we color the vertices as follows.
\begin{itemize}
    \item Color $x_k, x_{k+s(k - \ell)}, x_{k+2s(k - \ell)}, \dots, x_{k+(\lambda-1)s(k - \ell)}$ with color $k$ and remove these vertices from $\cP$,
    \item color the remaining vertices in turn with colors $1, \dots, k-1$. 
\end{itemize}
In the above second step, color the first vertex with color $1$. Suppose that we just colored the $i$th vertex with some color $j$. Then we color the next vertex with color $j+1$ if $j \le k-2$ and with color $1$ if $j=k-1$.

To get a proper coloring, it suffices to show that every edge of $\cP$ contains some vertex of color $k$. Indeed,   it holds for the first edge $e_1$ of $\cP$ and for all edges intersecting $e_1$ since $x_k$ lies in all those edges. Define the first vertex of the $i$th edge $e_i$ as $x_{f(i)}$, where $f(i) = (i-1)(k-\ell)+1$.
Let $i^* := s+ 1$ be the smallest integer so that $f(i^*) > k$, that is, the $i^*$th edge $e_{i^*}$ is the first edge not containing $x_k$. But the vertices of $e_{i^*}$ are $x_{s(k - \ell)+1}, \dots, x_{s(k - \ell)+k}$. So $e_{i^*}$ and all successive edges intersecting $e_{i^*}$ contain $x_{s(k - \ell)+k}$, a vertex of color $k$. Continue this process and we get a proper $k$-coloring.  
\end{proof}

Let $\cH$ be a $k$-partite $k$-graph with partition classes $V_1, V_2, \dots, V_k$.
Given $1\le \ell <k$, an $\ell$-path $\cP$ of length $\lambda s$ is called \emph{canonical} with respect to $(V_1, V_2, \dots, V_k)$ if
$|V_k\cap V(\cP)| = \lambda$ and the sizes of $V_1\cap V(\cP), V_2\cap V(\cP), \dots, V_{k-1}\cap V(\cP)$ are as equal as possible.




\begin{proposition}\label{pro:kl-edge}
Let $\eps>0$ and $0< b\le \frac{1}{k-1}$.
If an $n$-vertex $k$-partite $k$-graph $\cH$ has all parts of size at most $bn$ with $e(\cH)\ge \eps n^k$, then there exists a canonical $\ell$-path of length $\lambda s\ge \eps n$ in $\cH$.
\end{proposition}
\begin{proof}
If $\deg_{\cH}(S)<\eps n^{k-\ell}$ for an $\ell$-set $S\in \binom{V(\cH)}{\ell}$, then we remove all edges containing $S$.
We keep doing this process until $\deg(S)=0$ or $\deg(S)\ge \eps n^{k-\ell}$ for all $\ell$-sets $S$ in the present hypergraph.
Note that we have removed less than $\eps n^{k-\ell}\cdot (bn)^\ell=\eps b^\ell n^k$ edges from $\cH$, and the resulting $k$-graph $\cH'$ is non-empty since $e(\cH')>e(\cH)-\eps b^\ell n^k>0$.

Now we greedily find a canonical $\ell$-path in $\cH'$ using the degree condition.
Note that any $\ell$-path of length less than $\eps n+s$ can be extended to a longer path.
Indeed, suppose that we have an $\ell$-path $\cP=e_1e_2\cdots e_t$ of $\cH'$ with $t<\eps n+s$.
Let $X_t:=V(\cP)$ be the set of vertices used so far and thus $|X_t|= \ell+t(k-\ell)$. 
Take $S$ as the $\ell$-end of the path $\cP$ contained in $e_t$. 
Note that $\deg_{\cH'}(S)\ge \eps n^{k-\ell}$. 
To extend, we need to find another edge $e_{t+1}$ containing $S$ and $(e_{t+1}\setminus S)\cap X_t=\emptyset$.
The number of edges containing $S$ disjoint with $X_t\setminus S$  is at least 
\[\eps n^{k-\ell}-(\ell+t(k-\ell))\cdot (bn)^{k-\ell-1}> n^{k-\ell-1}\left(\eps n-(\ell + (\eps n+s) (k-\ell) )b^{k-\ell-1}\right)>0\]
for sufficiently large $n$ as $t< \eps n+s$ and $b\le \frac{1}{k-1}$.
Thus, there is an edge $e_{t+1}$ containing $S$ such that $(e_{t+1}\setminus S)\cap X_t=\emptyset$. 
Hence we obtain an $\ell$-path $\cP'=e_1e_2\cdots e_t e_{t+1}$ of $\cH'$ by adding $e_{t+1}$ to $\cP$.
Since this extension can be executed as long as $t<\eps n+s$, we can build a canonical $\ell$-path $\cP^*$ of length $\lambda s\ge \eps n$ in $\cH$.
\end{proof}


We use the following lemma to show that $(\eps,d)$-regular $k$-tuples $(V_1,V_2,\ldots,V_k)$ can be almost covered by a constant number of vertex-disjoint $\ell$-paths.

\begin{lemma}\label{lem:bal-paths}
    Fix $k\ge 3,1\le \ell<k$ such that $(k-\ell)\nmid k$, and $\eps,d>0$ such that $d\ge2\eps$.
    Let $m>\frac{2sk}{\eps^{k+2}}$.
    Suppose $\mathcal{V}:= (V_1, V_2, \dots, V_k)$ is an $(\eps, d)$-regular $k$-tuple with
    \[|V_1|=\cdots= |V_{k-1}|=(sk-s\ell-1)m \text{ and } |V_k|=(k-1)m.\]
    Then there exists a family of at most $\frac{sk}{\eps^{k+1}}$ vertex-disjoint canonical $\ell$-paths that together cover all but at most $2sk^2\eps m$ vertices of $\mathcal V$.
\end{lemma}

\begin{proof}
We greedily find vertex-disjoint canonical $\ell$-paths by Proposition~\ref{pro:kl-edge} in $\mathcal{V}$ until less than $2\eps |V_k|$ vertices are uncovered in $V_k$.
Suppose that we have found $q$ vertex-disjoint canonical $\ell$-paths $\cP_1,\cP_2,\ldots,\cP_q$ for some $q>0$, each of length $p_i\ge \eps^{k+1} m$ with $p_i\in s \mathbb N$.
For $i\in [k]$, 
let $U_i$ be the set of uncovered vertices in $V_i$
and suppose that $|U_k|\ge 2\eps |V_k|$.
As each path $\cP_i$ is canonical, we have that $|V(\cP_i)\cap V_k| = p_i/s$ and for $j\in [k-1]$, $|V(\cP_i)\cap V_j| \le \left\lceil\frac{\ell+p_i(k-\ell)-p_i/s}{k-1}\right\rceil\le \frac{p_i}{s} \frac{sk-s\ell-1}{k-1}+2$.
Because $|V_k| - |U_k|=\sum_{i\in [q]}|V(\cP_i)\cap V_k|=\sum_{i\in [q]}p_i/s \le (1-2\eps)(k-1)m$, we have for $j\in [k-1]$,
\[
\sum_{i\in [q]}|V(\cP_i)\cap V_j|\le 2q+\sum_{i\in [q]} \frac{p_i}{s} \frac{sk-s\ell-1}{k-1}\le (1-2\eps)|V_j|+2q.
\]
As $|V_k\cap V(\cP_j)|\ge {\eps^{k+1} m}/{s}$ for $j\in[q]$, we derive $q\le \frac{(k-1)m}{{\eps^{k+1} m}/{s}}\le \frac{sk}{\eps^{k+1}}<\eps m/2$.
Therefore, we have for $j\in [k-1]$, $|U_i|\ge 2\eps |V_j|-2q\ge \eps |V_j|$.
Now we consider a $k$-partite subhypergraph $\mathcal{V'}$ with partition classes $U_1,U_2,\ldots,U_k$. 
By regularity, $e(\mathcal{V'})\ge (d-\eps)\prod_{i=1}^k|U_i|\ge (d-\eps)\eps^k m^k\ge \eps^{k+1}m^k$, so we can apply Proposition~\ref{pro:kl-edge} and find a canonical $\ell$-path of length $p_{q+1}\ge \eps^{k+1} m$.
Let $\cP_1,\cP_2,\ldots,\cP_{q'}$ denote the canonical $\ell$-paths obtained in $\mathcal{V}$ after the iteration stops.
Note that as each path $\cP_i$ is canonical, we have that $|V(\cP_i)\cap V_k| = p_i/s$ and for $j\in [k-1]$, $|V(\cP_i)\cap V_j| \ge \left\lfloor\frac{\ell+p_i(k-\ell)-p_i/s}{k-1}\right\rfloor \ge \frac{p_i}{s} \frac{sk-s\ell-1}{k-1}-1$.
Because $\sum_{i\in [q']}|V(\cP_i)\cap V_k|=\sum_{i\in [q']}p_i/s \ge (1-2\eps)(k-1)m$, we have for $j\in [k-1]$,
\[
\sum_{i\in [q']}|V(\cP_i)\cap V_j|\ge -q'+\sum_{i\in [q']} \frac{p_i}{s} \frac{sk-s\ell-1}{k-1}\ge (1-2\eps)|V_j|-q'.
\]
Moreover, as $|V_k\cap V(\cP_i)|\ge {\eps^{k+1} m}/{s}$ for $i\in[q']$, we derive $q'\le \frac{(k-1)m}{{\eps^{k+1} m}/{s}}\le \frac{sk}{\eps^{k+1}}$.
Since $m>\frac{2sk}{\eps^{k+2}}$, we have $(k-1)q'<\frac{sk^2}{\eps^{k+1}}<sk^2\eps m$.
Therefore, the total number of uncovered vertices in $\mathcal{V}$ is at most $2\eps |V_k|+\sum_{j=1}^{k-1}(2\eps |V_j|+q')=(sk-s\ell)(k-1)2\eps m+(k-1)q'\le 3sk^2\eps m$.
\end{proof}

\subsection{Proof of Theorem~\ref{thm:nonextremal}}
For the non-extremal case in~\cite{HZ15}, Han and Zhao used the absorbing method pioneered by R\"odl, Ruci\'nski and Szemer\'edi~\cite{RRS2006}.
In fact, the absorbing lemma and the reservoir lemma used in~\cite{HZ15} are due to K\"uhn, Mycroft and Osthus~\cite[Lemmas 6.3 and 8.1]{KMO10} which in fact hold for all $k\ge 3$ and $1\le \ell <k$.
Therefore, to prove Theorem~\ref{thm:nonextremal}, it suffices to 
show that the following result, which extends the path-cover lemma in~\cite[Lemma 2.3]{HZ15}.
\begin{lemma}\emph{(Path-cover lemma)}\label{lem:path cover}
    For all integers $k \geq 3$, $1\le \ell <k$ such that $(k-\ell)\nmid k$, and every $\alpha, \gamma > 0$, there exist integers $p$ and $n_0$ such that the following holds. Let $\cH$ be a $k$-graph on $n > n_0$ vertices with $\delta_{k-1}(\cH) \geq \left( \frac{1}{s(k - \ell)} - \gamma\right) n,$
    then there is a family of at most $p$ vertex-disjoint $\ell$-paths that together cover all but at most $\alpha n$ vertices of $\cH$, or $\cH$ is $12k^3\gamma$-extremal.
\end{lemma}

The proof of our Theorem~\ref{thm:nonextremal}, using our Lemma~\ref{lem:path cover} above, is identical to that of the non-extremal case \cite[Theorem~1.4]{HZ15}, and is thus omitted.



\begin{proof}
[Proof of Lemma~\ref{lem:path cover}]
Fix integers $k\ge 3,1\le \ell<k$ with $(k-\ell)\nmid k$
and $0<\alpha,\gamma<1$.
Let $\Delta=10k^3\gamma$ and $\beta=\alpha/2$.
Let $T_0$ be the constant returned by applying Proposition~\ref{cor:re-partition} with $c=\frac{1}{s(k-\ell)}-\gamma$, $d=\gamma/2, \eps\le \min \{\frac{\gamma^2}{16}$, 
$\frac{\alpha}{3s^4k}\}$ and $t_0>4k/\gamma$.
We let $p=\frac{sT_0}{k^k\eps^{k+1}}$.

Let $n$ be sufficiently large.
Applying Proposition~\ref{cor:re-partition} to $\cH$ with the constants chosen above, we obtain an $\eps$-regular partition $V=V_0\cup V_1\cup \cdots \cup V_t$ such that the partition classes satisfy 
\[
|V_1|=|V_2|=\cdots=|V_t|=(sk-s\ell-1)(k-1)m=:N
\]
for some positive integer $m$ by possibly moving at most $t(sk-s\ell-1)(k-1)\le \eps n$ vertices to $V_0$ and then $(1-2\eps)\frac{n}{t}\leq N\leq \frac{n}{t}$. 
Moreover, we also obtain a cluster hypergraph $\cR=\cR(\eps,d)$ on $t\ge t_0$ vertices such that for all but at most $\sqrt{\eps}t^{k-1}$ $(k-1)$-sets $S\subseteq V(\cR)$ satisfy
\[
\deg_\cR(S)\geq \left(\frac{1}{s(k-\ell)}-\gamma-d-\sqrt{\eps}\right)t-(k-1)\geq \left(\frac{1}{s(k-\ell)}-2\gamma\right)t,
\]
where the second inequality follows from $d=\gamma/2, \sqrt{\eps}\le\gamma/4$ and $k-1<\gamma t_0/4\le \gamma t/4$.
Let $\mathcal{F}_{k,\ell}$ be the $k$-graph whose vertex set is the disjoint union of sets $A_1, \dots, A_{sk-s\ell-1}$ and $B$ of size $k-1$ and whose edges are all the $k$-sets of the form $A_i \cup \{b\}$ (for all $i\in [s(k-\ell)-1]$ and all $b \in B$). 
Applying Lemma~\ref{lem:GHZ-2} to $\cR$ with the constants chosen above and
$\gamma_3=2\gamma,~K:=\mathcal{F}_{k,\ell}$,
we derive that either $\cR$ has a $\mathcal{F}_{k,\ell}$-tiling $\mathscr{F}$ that covers all but at most $\beta t$ vertices or $\cR$ is $\Delta$-extremal.

In the latter case,
there exists a set $B\subseteq V(\cR)$ such that $|B|=\left\lfloor \left(1-\frac{1}{s(k - \ell)}\right) t \right\rfloor$ and $e_{\cR}(B)\le \Delta t^k$.
Let $B'\subseteq V(\cH)$ be the union of the clusters in $B$. 
By regularity,
\begin{eqnarray}\label{eqn:edges}
e_{\cH}(B')\leq e_\cR(B)\cdot N^k+\binom{t}{k}\cdot d\cdot N^k+\eps \binom{t}{k}\cdot N^k+t\binom{N}{2}\binom{n-2}{k-2},  
\end{eqnarray}
where the right-hand side bounds the number of edges from regular $k$-tuples with high density, edges from regular $k$-tuples with low density, edges from irregular $k$-tuples and edges that lie in at most $k-1$ clusters.
Since $\eps<\gamma/16$ and $t\ge t_0>4k/\gamma$,
\eqref{eqn:edges} is at most 
\[
10k^3\gamma t^k\cdot\left(\frac{n}{t}\right)^k+\binom{t}{k}\cdot\frac{\gamma}{2}\cdot\left(\frac{n}{t}\right)^k+\frac{\gamma}{16}\cdot\binom{t}{k}\cdot\left(\frac{n}{t}\right)^k+t\binom{n/t}{2}\binom{n-2}{k-2}<11k^3\gamma n^k.
\]
Note that 
$|B'| = \left\lfloor \left(1-\frac{1}{s(k - \ell)}\right) t \right\rfloor N 
\le \left(1-\frac{1}{s(k - \ell)}\right) t \cdot \frac{n}{t} 
= \left(1-\frac{1}{s(k - \ell)}\right) n,$
and consequently 
$|B'| \le \left\lfloor \left(1-\frac{1}{s(k - \ell)}\right) n \right\rfloor.$
On the other hand,
\begin{align*}
|B'| &=  \left\lfloor \left(1-\frac{1}{s(k - \ell)}\right) t \right\rfloor N \ge \left(\left(1-\frac{1}{s(k - \ell)}\right)t-1 \right)\cdot(1-2\eps)\frac{n}{t} \\
&= \left( \left(1-\frac{1}{s(k - \ell)}\right)t -2\eps t +2\eps \frac{1}{s(k - \ell)}t-1 \right)\frac{n}{t} \ge \left(  \left(1-\frac{1}{s(k - \ell)}\right)t - 2\eps t \right)\frac{n}{t} \\
&> \left(1-\frac{1}{s(k - \ell)}\right) n - 2\eps n.
\end{align*}
By adding at most $2\eps n$ vertices from $V \setminus B'$ to $B'$, we obtain a set $B'' \subseteq V(\mathcal{H})$ of size exactly $\left\lfloor \left(1-\frac{1}{s(k - \ell)}\right) n \right\rfloor$
with 
$e(B'') \le e(B') + 2\eps n \cdot n^{k-1} < 12k^3\gamma n^k.$
Hence $\mathcal{H}$ is $12k^3\gamma$-extremal.

In the former case, 
the union of the clusters covered by $\mathscr{F}$ contains all but at most $\beta tN + |V_0|\le \alpha n/2+2\eps n $ vertices of $\mathcal{H}$.
Let $\mathcal{F}$ be an arbitrary copy $\mathcal{F}_{k,\ell}$ in the $\mathcal{F}_{k,\ell}$-tiling $\mathscr{F}$ of $\cR$ with the vertex set $V(\mathcal{F})$ grouped into sets $B,A_1,\ldots,A_{s(k-\ell)-1}$, all of the same size $k-1$.
Clearly, $|V(\mathcal{F})|=s(k-\ell)(k-1)$.
The edges of $\mathcal{F}$ are the sets $\{b\}\cup A_i$ with $i\in [s(k-\ell)-1]$ and $b\in B$.
Now we let $W_1,\ldots,W_{k-1}$ denote the corresponding clusters of $k-1$ vertices of $B$ in $\cH$ and let $W_k,\ldots,W_{s(k-\ell)(k-1)}$ denote the corresponding clusters of vertices of $A_1,\ldots,A_{s(k-\ell)-1}$ in $\cH$.
Now we split each $W_i$, $i\in [k-1]$, into $s(k-\ell)-1$ disjoint sets $W^1_i, \ldots, W^{s(k-\ell)-1}_i$ of equal size and split each $W_j$, $j\in \{k,\dots, s(k-\ell)(k-1)\}$, into $k-1$ disjoint sets $W^1_j,\ldots,W^{k-1}_j$ of equal size.
Let $\eps_\ell:=(sk-s\ell-1)\eps$.
Then each of the $k$-tuples $\left(W_i^j,W_{j(k-1)+1}^i,W_{j(k-1)+2}^i,\ldots,W_{(j+1)(k-1)}^i\right)$ with $i\in [k-1]$, $j\in [s(k-\ell)-1]$
is $(\eps_\ell,d)$-regular and of sizes $(k-1)m, (sk-s\ell-1)m, (sk-s\ell-1)m, \ldots, (sk-s\ell-1)m$.
Applying Lemma~\ref{lem:bal-paths} to these $(sk-s\ell-1)(k-1)$ $k$-tuples, we find a family of at most 
$\frac{sk}{\eps_\ell^{k+1}}$
vertex-disjoint canonical $\ell$-paths in each $k$-tuple covering all but at most $3s^2k\eps_\ell m$ vertices.
Since $|\mathscr{F}|\le \frac{t}{s(k-\ell)(k-1)}$, we thus obtain a path-tiling that consists of at most 
\[\frac{t}{s(k-\ell)(k-1)}\cdot (sk-s\ell-1)(k-1)\cdot\frac{sk}{\eps_\ell^{k+1}}\le \frac{skt}{(sk-s\ell-1)^{k+1}\eps^{k+1}}\le \frac{sT_0}{k^k\eps^{k+1}}=
p\] 
$\ell$-paths whose union covers all but at most $
\frac{t}{s(k-\ell)(k-1)}\cdot (sk-s\ell-1)(k-1)\cdot 3s^2k\eps_\ell m+
\alpha n/2 + 2\eps n\le (3s^2k+2)\eps n +\alpha n/2
\le \alpha n$ vertices of $\cH$ where $\eps_\ell mt\le \eps Nt\le\eps n$
and 
$\eps\le \frac{\alpha}{3s^4k}$.
This completes the proof.
\end{proof}



\section{Proof of Theorem~\ref{thm:extremal}}\label{sec:3}
In this section we provide the proof for the extremal case.
Indeed, under our approach (and similar to previous proofs), there is only one place that the exact minimum co-degree condition is needed, where we use the degree condition to find a small collection of vertex-disjoint $\ell$-paths with certain constraints.
We then write our proof in a unified way: we can build a Hamilton $\ell$-cycle if we assume the existence of such collection of paths (Theorem~\ref{thm:main result}), and then show how to build such collection of paths under either of our minimum co-degree conditions (Lemma~\ref{lem:sq paths} and Lemma~\ref{lem:sq}).

Recall that $s:=\lceil \frac{k}{k-\ell}\rceil$. 
Suppose $k\geq 3$ 
and $0<\Delta \ll 1$. 
Let $n\in (k-\ell)\mathbb{N}$ be sufficiently large and $\cH=(V,E)$ be a $k$-graph with $n$ vertices such that 
$\delta_{k-1}(\cH)\geq \frac{n}{s(k-\ell)}$ for a moment.

Assume that $\cH$ is $\Delta$-extremal, i.e., there exists a set $B\subseteq V$ such that $|B| = \left\lfloor \left(1-\frac{1}{s(k - \ell)}\right) n \right\rfloor$ and $e(B)\leq \Delta n^{k}$. Let $A=V\setminus B$, and then $|A|=\lceil \frac{n}{s(k-\ell)}\rceil$.
For the convenience of later calculations, we let $\varepsilon_{0}=2k!\Delta\left(1-\frac{1}{s(k-\ell)}\right)^{-k} \ll1$. 
Indeed, we have
\begin{eqnarray}\label{i1}
\begin{aligned}
e(B)\leq \Delta n^{k}= \varepsilon_{0}\left(1-\frac{1}{s(k-\ell)}\right)^{k}\frac{n^{k}}{2k!}\leq \varepsilon_{0}\binom{|B|}{k}.
\end{aligned}
\end{eqnarray}

Now let us include some additional notation for this section.
Given a $k$-graph $\cH=(V,E)$ and two vertex sets $X,Z\subseteq V$, $|X|<k$, let $\deg(X, Z)$ be the number of $(k-|X|)$-sets $Y\subseteq Z$, the \emph{neighbor} of $X$, satisfying $X\cup Y\in E$. 
The number of \emph{non-edges} on $X\cup Z$ containing $X$ is defined as $\overline{\deg}(X,Z)=\binom{|Z\setminus X|}{k-|X|}-\deg(X, Z)$.
Given two disjoint vertex sets $X,Y$ and two integers $i,j\geq 0$, we say that a set $S\subseteq X\cup Y$ is an $X^{i}Y^{j}$-set if $|S\cap X|=i$ and $|S\cap Y|=j$. 
In particular, when two disjoint vertex sets $X,Y\subseteq V$ and $i+j=k$, we define $\cH(X^{i}Y^{j})$ to be the family of all edges of $\cH$ that are $X^{i}Y^{j}$-sets and $e(X^{i}Y^{j})=|\cH(X^{i}Y^{j})|$.
The number of non-edges among $X^{i}Y^{k-i}$-sets is denoted by $\overline{e}(X^{i}Y^{k-i})$. 
In addition, given a set $L\subseteq X\cup Y$ with $|L\cap X|=k_1\leq i$ and $|L\cap Y|=k_2\leq k-i$, denote by $\deg(L,X^{i}Y^{k-i})$ the number of edges in $\cH(X^{i}Y^{k-i})$ that contain $L$, and
$\overline{\deg}(L,X^{i}Y^{k-i})=\binom{|X|-k_1}{i-k_1}\binom{|Y|-k_2}{k-i-k_2}-\deg(L,X^{i}Y^{k-i})$. 
For a vertex $v\in V$, the \emph{link hypergraph} of $v$ in $\cH$ is the $(k-1)$-graph $\cL_v=(V\setminus\{v\},E_v)$ where $E_v=\{e\setminus\{v\}:e\in E,v\in e\}$.

\subsection{Preparation} 
Let $\varepsilon_{1}=\varepsilon_{0}^{1/4}$and $\varepsilon_{2}=2\varepsilon_{1}^{2}$. 
Suppose that the partition $V=A\cup B$ satisfies that $|B| = \left\lfloor \left(1-\frac{1}{s(k - \ell)}\right) n \right\rfloor$ and~\eqref{i1}. In addition, assume that $e(B)$ is the smallest among all such partitions. We now classify the vertices further by defining
\begin{eqnarray}\label{eqn:vxt part}
\begin{aligned}
A':&=\left\{v\in V:\deg(v,B)\geq (1-\varepsilon_{1})\binom{|B|}{k-1}\right\},\\
B':&=\left\{v\in V:\deg(v,B)\leq \varepsilon_{1}\binom{|B|}{k-1}\right\},\\
V_{0}:&=V\setminus(A'\cup B').
\end{aligned}
\end{eqnarray}

Moreover, one important step is to connect several short $\ell$-paths to a longer path. 
To make this easy, we classify the $\ell$-sets in $B'$: they are good to us if they have high degree in the $k$-graph $\cH[A'(B')^{k-1}]$. 
Formally, for $0\leq \ell'\leq \ell$, an $\ell'$-set $L'\subseteq V$ is $\varepsilon_{1}$-\emph{bad} if $\deg(L',B)>\varepsilon_{1} \binom{|B|}{k-\ell'}$, otherwise $\varepsilon_{1}$-\emph{good}. 
Then we define that an ordered $\ell$-set ${\bf L}=(v_1, v_2, \dots, v_{\ell})$ is \emph{linkable} 
if all the ordered $(\ell-(i-1)(k-\ell))$-sets $(v_{(i-1)(k-\ell)+1}, v_{(i-1)(k-\ell)+2}, \dots , v_{\ell})$ are $\eps_1$-good for $i\in [s-1]$,
otherwise \emph{non-linkable}.
In particular, we often write $L=\{v_1,v_2,\ldots,v_\ell\}$ for the $\ell$-set of ${\bf L}$ in this paper.

Now we are ready to state our technical result, which says that we can construct a Hamilton $\ell$-cycle given the existence of certain collection of vertex-disjoint $\ell$-paths.



\begin{theorem}\label{thm:main result}
Let integer $k>\ell\ge 2$ such that $(k-\ell)\nmid k$ and $s_*:=s_*(k)$ be a function of $k$.
Let $\cH$ be an $n$-vertex $k$-graph such that $n\in (k-\ell)\mathbb{N}$ is sufficiently large.
Suppose $A',B',V_0$ is defined as in~\eqref{eqn:vxt part}. 
Let $q=|A\cap B'|$.
If $\cH$ has $sq$ vertex-disjoint $\ell$-paths $P_1,P_2,\ldots,P_{sq}$, each of length at most $s_*$ and each containing two linkable $\ell$-ends, then $\cH$ contains a Hamilton $\ell$-cycle.
\end{theorem}

Most of the rest of the paper will be devoted to the proof of Theorem~\ref{thm:main result}, for which we need some preparations.
After we prove Theorem~\ref{thm:main result}, we show how our minimum co-degree conditions imply the existence of the desired collection of paths.

\subsection{Useful tools}

The following results from ~\cite[Claims 3.1 and 3.2]{HZ15} concerning relations on vertex partitions are very useful for our proof.
\begin{lemma}\label{lem1}\emph{\cite{HZ15}}
$A\cap B'\neq \emptyset$ implies that $B \subseteq B'$, and $B\cap A'\neq \emptyset$ implies that $A \subseteq A'$.
\end{lemma}

\begin{lemma}\label{11}\emph{\cite{HZ15}}
$\{|A\setminus A'|,|B\setminus B'|,|A'\setminus A|,|B'\setminus B|\} \leq \varepsilon_{2}|B|$ and $|V_{0}|\leq 2\varepsilon_{2}|B|$.
\end{lemma}

We omit their (simple) proofs which are identical to that in~\cite{HZ15}.
Next we give several properties related to linkable $\ell'$-sets.

\begin{fact}\label{c0}
The number of $\beta$-bad $\ell'$-sets in $B$ is at most $\frac{\eps_0}{\beta}\binom{|B|}{\ell'}$.
\end{fact}
\begin{proof}
By~\eqref{i1}, the number of $\beta$-bad $\ell'$-sets in $B$ is at most
\begin{equation*}
\frac{\varepsilon_0\binom{|B|}{k}\cdot \binom{k}{\ell'}}{\beta \binom{|B|}{k-\ell'}}=\frac{\eps_0}{\beta}\binom{|B|-k+\ell'}{\ell'} \leq \frac{\eps_0}{\beta}\binom{|B|}{\ell'},
\end{equation*}
where the first equality holds because ${a \choose b}{b\choose c}={a \choose b-c}{a-b+c\choose c}$. 
\end{proof}

We frequently need the following simple fact. 

\begin{fact}\label{fact:no bad}
Let $d\in \mathbb{N}$. 
There are at least 
$(1-2^{d}\eps_0/\beta)\binom{|B|}{d}$ $d$-sets in $B$ containing no $\beta$-bad subset.   
\end{fact}
\begin{proof}
To prove the result, we first randomly pick a subset $K'\subseteq B$ with $|K'|\leq d$. 
Then by Fact~\ref{c0}, 
the probability that $K'$ is $\beta$-bad is at most  ${\frac{\eps_0}{\beta}\binom{|B|}{|K'|}}/{\binom{|B|}{|K'|}}=\varepsilon_0/\beta$. By the union bound, the probability that there is no $\beta$-bad subset of any $d$-set of $B$ is at least $1-2^d\eps_0/\beta$.    
\end{proof}

Our next proposition exploits the property of $\eps_1$-good sets.

\begin{proposition}\label{c1}
Every $\varepsilon_{1}$-good $\ell'$-set $L'\subseteq B'$ satisfies $\deg(L',A'(B')^{k-1})\geq (1-2sk\varepsilon_{1}) |A'|\binom{|B'|-\ell'}{k-1-\ell'}$.
\end{proposition}
\begin{proof}
Given an $\varepsilon_{1}$-good $\ell'$-set $L'\subseteq B'$, we get
\begin{eqnarray}\label{eqd}
\begin{aligned}
\sum_{L' \subseteq D \subseteq B', |D| = k-1} \deg(D,V) = \sum_{L' \subseteq D \subseteq B', |D|=k-1} (\deg(D, A') + \deg(D, B') + \deg(D, V_0)) . \\
\end{aligned}
 \end{eqnarray}
By~\eqref{eqn:conj value}, the left-hand side above is at least
$\binom{|B'|-\ell'}{k-1-\ell'}|A|$. On the other hand,
\begin{eqnarray*}
\begin{aligned}
\sum_{L' \subseteq D \subseteq B', |D|=k-1} (\deg(D, B') + \deg(D, V_0)) \leq (k - \ell') \deg(L', B') + \binom{|B'|-\ell'}{k-1-\ell'}|V_0| .
\end{aligned}
 \end{eqnarray*}
Since $L'$ is $\varepsilon_{1}$-good and $|B'\setminus B|\leq \varepsilon_{2}|B|$, we have
\begin{eqnarray*}
\deg(L', B') &\leq& \deg(L', B) + |B' \setminus B| \binom{|B'|-1}{k-1-\ell'}\\
&\leq& \varepsilon_1 \binom{|B|}{k - \ell'} + \varepsilon_2 |B| \binom{|B'|-1}{k-1-\ell'}.
 \end{eqnarray*}
Thus, we have 
\begin{eqnarray*}
(k - \ell') \deg(L', B') &\leq&  \varepsilon_{1} |B| \binom{|B|-1}{k-1-\ell'}+ (k -\ell') \varepsilon_{2} |B| \binom{|B'|-1}{k-1-\ell'}\\
&\leq& 2\varepsilon_{1} |B| \binom{|B'| - \ell'}{k-1-\ell'}
 \end{eqnarray*}
as $\varepsilon_{2} \ll \varepsilon_{1}$ and $||B|-|B'||\le \eps_2|B|$.
Then, combining \eqref{eqd} and Lemma~\ref{11}, we derive that 
\begin{eqnarray*}
\sum_{\substack{L' \subseteq D \subseteq B',|D| = k - 1}} \deg(D, A') &\ge& \binom{|B'| - \ell'}{k-1-\ell'} (|A| - |V_0|) - 2\varepsilon_{1} |B| \binom{|B'| - \ell'}{k-1-\ell'} \\
&\ge& \binom{|B'| - \ell'}{k-1-\ell'} \big(|A'| - 3\varepsilon_{2} |B| - 2\varepsilon_{1} |B|\big).
 \end{eqnarray*}
Since $|B| \le (sk - s\ell - 1)|A| \le (sk - s\ell)|A'|$, we obtain that
\begin{align*}
\deg(L', A'(B')^{k-1}) =& \quad \sum_{\substack{L' \subseteq D \subseteq B',|D| = k - 1}} \deg(D, A')\\
\geq& \quad \binom{|B'| - \ell'}{k-1- \ell'}\left(1 - (sk-s\ell)(3\varepsilon_{2} + 2\varepsilon_{1})\right) |A'| \\
\geq& \quad (1-2sk\varepsilon_{1})|A'|\binom{|B'| - \ell'}{k-1- \ell'}. \qedhere
\end{align*}

\end{proof}

In the next claim we show that we can connect any two disjoint linkable $\ell$-sets of $B'$ by an $\ell$-path of length $2s$ while avoiding any set of $\frac{n}{2s(k-\ell)}$ vertices of $V$. 
We use $L_i$ to denote the underlying $\ell$-set of the linkable $\ell$-set ${\bf L}_i$ for $i\in\{0,1\}$.

\begin{claim}\label{c2}
Given two disjoint ordered linkable $\ell$-sets ${\bf L}_0, {\bf L}_1$ in $B'$ and a vertex set $U\subseteq V$ with $|U|\leq \frac{n}{2s(k-\ell)}$, there exist two vertices $x, y\in A'\setminus U$ and a $(2sk-(2s+1)\ell-2)$-set $C\subseteq B'\setminus U$ such that $L_0\cup L_1\cup \{x, y\}\cup C$ spans an $\ell$-path of length $2s$ with ends ${\bf L}_0', {\bf L}_1'$, where ${\bf L}_i'$ is the ordered $\ell$-set obtained by reversing the ordering of ${\bf L}_i$ for $i\in\{0,1\}$.
\end{claim}
\begin{proof}

Let ${\bf L}_0=(v_{2sk-2s\ell+1}, v_{2sk-2s\ell+2},\ldots,v_{2sk-2s\ell+\ell})$ and ${\bf L}_1=(v_{1}, v_{2},\ldots,v_{\ell})$ in $B'$. 
Our goal is to choose vertices $v_{\ell+1},\dots, v_{2sk-2s\ell}$ such that 
$e_{i}=\{v_{1+(i-1)(k-\ell)},v_{2+(i-1)(k-\ell)},\ldots,v_{k+(i-1)(k-\ell)}\}$ with $i\in[2s]$ 
are edges of $\cH$, that is, the $e_i$'s form an $\ell$-path of length $2s$.
Indeed, we randomly choose two vertices $v_k, v_{(2s-1)(k-\ell)+1}\in A'\setminus U$ and other $2sk-(2s+1)\ell-2$ vertices from $B'\setminus U$ to form 
a (random) $(A')^{2}(B')^{2sk-(2s+1)\ell-2}$-set $D$.
Let $C:=D\setminus \{v_{k}, v_{(2s-1)(k-\ell)+1}\}\subseteq B'\setminus U$. 
Define ${\bf L}_0'=(v_{2sk-2s\ell+\ell}, v_{2sk-2s\ell+\ell-1},\ldots,v_{2sk-2s\ell+1})$ and ${\bf L}_1'=(v_{\ell}, v_{\ell-1},\ldots,v_{1})$.
To prove that there is a such $D$ such that $L_0\cup L_1\cup D$ spans an $\ell$-path of length $2s$, by the union bound, it suffices to show that $\mathbb{P}[e_{i}\notin E(\cH) ] <\frac{1}{2s}$.

Define that $\ell_{i}=|e_{i}\cap L_1|$ if $i\in[s]$, and $\ell_{i}=|e_{i}\cap L_0|$ if $i\in[2s]\setminus [s]$. 
Since $v_k, v_{(2s-1)(k-\ell)+1}$ are chosen from $ A'\setminus U$ randomly and the vertices of $C$ are chosen from $B'\setminus U$ randomly, they follow a hypergeometric distribution.
Thus, by Proposition~\ref{c1}, for $i\in [2s]$, 
we have
\begin{equation}\label{eqn:prob}
\mathbb{P}[e_{i}\notin E(\cH)]
\leq \frac{2sk\varepsilon_{1}|A'|\binom{|B'|-\ell_{i}}{k-1-\ell_{i}}(k-1-\ell_{i})!}{|A'\setminus U|\binom{|B'\setminus U|}{k-1-\ell_{i}}(k-1-\ell_{i})!}. 
\end{equation}

Since $|A|=\lceil \frac{n}{s(k-\ell)}\rceil$, by Lemma~\ref{11}, we obtain that
$|U|\leq \frac{n}{2s(k-\ell)}\leq \frac{|A|}{2}<\frac{2}{3}|A'|$
 and \[|U|\leq \frac{n}{2s(k-\ell)}\leq \frac{|B|+1}{2(sk-s\ell-1)}\leq \frac{(1+2\varepsilon_{2})|B'|}{2s(k-\ell)-2}\leq \frac{(1+2\varepsilon_{2})|B'|}{2(k-1)}<\frac{|B'|}{k-1}-k.\]
Thus $|A'\setminus U|>\frac{|A'|}{3}$ and $|B'\setminus U|>\frac{k-2}{k-1}|B'|+k$.
Then \eqref{eqn:prob} can be upper bounded by
\[\frac{2sk\varepsilon_{1}|A'|\binom{|B'|}{k-1-\ell_{i}}}{\frac{|A'|}{3}\binom{\frac{k-2}{k-1}|B'|+k}{k-1-\ell_{i}}}\leq \frac{6sk\varepsilon_{1}|B'|^{k-1-\ell_{i}}}{\left(\frac{k-2}{k-1}|B'|\right)^{k-1-\ell_{i}}}\leq 6sk\varepsilon_{1}\left(\frac{k-1}{k-2}\right)^{k-1}<\frac{1}{2s}\]
as $\varepsilon_{1}\ll\frac{1}{sk^{k}}$.
The proof is completed.
\end{proof}



Next we show that we can extend an $\ell$-path from a linkable $\ell$-end with $s-1$ edges of $\cH$ by adding one vertex from $A'$ and $(s-1)(k-\ell)-1$ vertices from $B'$ such that the new $\ell$-end is also linkable. 
This is possible even if we need to avoid at most $\frac{n}{2s(k-\ell)}$ vertices in $V$.
It is useful to us for adjusting the ratio of $|A'|$ and $|B'|$ -- indeed, we use it when $|A'|/|B'| > 1/(s-1)$.

\begin{claim}\label{claim:extend l-set}
 Given an ordered linkable $\ell$-set ${\bf L}\subseteq B'$ and a vertex set $U\subseteq V$ with $|U|\le \frac{n}{2s(k-\ell)}$, there exists an $A'(B')^{(s-1)(k-\ell)-1}$-set $W\subseteq V\setminus U$ such that $L\cup W$ spans an $\ell$-path of length $s-1$ and
 the new $\ell$-end of the path is linkable.
\end{claim}
\begin{proof}
Let ${\bf L}=(v_1,v_2,\ldots,v_\ell)$ 
and $L=\{v_1,v_2,\ldots,v_\ell\}$
in $B'$.
Our goal is to choose vertices $v_{\ell+1},\ldots,v_{\ell+(s-1)(k-\ell)}$ such that 
$e_{i}=\{v_{1+(i-1)(k-\ell)},v_{2+(i-1)(k-\ell)},\ldots,v_{k+(i-1)(k-\ell)}\}$ with $i\in[s-1]$ are edges of $\cH$ (i.e., the $e_i$'s form an $\ell$-path of length $s-1$), and the new $\ell$-end of the path is linkable.
Indeed, 
we randomly choose a vertex $v_{\ell+1}\in A'\setminus U$ and other $(s-1)(k-\ell)-1$ vertices from $B'\setminus U$ to form 
a (random) $A'(B')^{(s-1)(k-\ell)-1}$-set $W$.
Let $S:=W\setminus \{v_{\ell+1}\}\subseteq B'\setminus U$.
Define that $\ell_{i}=|e_{i}\cap L|$ with $i\in[s-1]$.
Since $v_{\ell+1}$ is chosen from $A'\setminus U$ randomly and the vertices of $S$ are chosen from $B'\setminus U$ randomly, they follow a hypergeometric distribution.
Note that the distribution and the event $e_i \notin E(\cH)$ are the same as in Claim~\ref{c2}.
Thus, by the proof of Claim~\ref{c2}, we have $\mathbb{P}[e_{i}\notin E(\cH)]<\frac{1}{2s}$.
By the union bound, we know that the probability that all edge $e_i$'s form $\ell$-path of length $s-1$ is at least $1-\frac{s-1}{2s}>\frac{1}{2}$.
By Lemma~\ref{11} and Fact~\ref{fact:no bad}, 
the probability that $(v_{(s-1)(k-\ell)+1},\ldots,v_{(s-1)(k-\ell)+\ell})$ 
is non-linkable is at most $2^\ell\eps_2+\ell\cdot \frac{|B'\setminus B|}{|B'|}\leq (2^\ell+\ell)\eps_2$.
Hence, with probability $1-\frac{1}{2}-(2^\ell+\ell)\eps_2>0$, $L\cup W$ spans an $\ell$-path of length $s-1$ and the new $\ell$-end of the path is linkable.
\end{proof}

\subsection{Proof of Theorem~\ref{thm:main result}}\label{sec:ext}

In this subsection we prove Theorem~\ref{thm:main result}.
The work is split into two steps: we first connect the vertex-disjoint paths (with linkable ends) to a single (yet short) one that also covers the vertices of $V_0$, then we extend this single path to a Hamilton $\ell$-cycle solely using $A'(B')^{k-1}$ edges, via Lov\'asz Local Lemma.


\subsubsection{Building a short $\ell$-path $\cQ$}
 
For the first step, we connect the $sq$ vertex-disjoint $\ell$-paths with linkable $\ell$-ends given in Theorem~\ref{thm:main result} by applying Claim \ref{c2} repeatedly to build a short 
$\ell$-path $\cQ$ while also covers the vertices of $V_0$ -- indeed, we first put them into short paths. 
Note that we denote the underlying $\ell$-sets of linkable $\ell$-sets ${\bf L}_0$ and ${\bf L}_1$ by $L_0$ and $L_1$, respectively.
\begin{claim}\label{cover v0}
There exists a non-empty $\ell$-path $\cQ$ in $\cH$ with the following properties:
\begin{itemize}
    \item $V_{0}\subseteq V({\cQ})$,
    \item $|V({\cQ})|\leq (2ss_*+6s^2)k\eps_2|B|$,
    \item the two $\ell$-ends ${\bf L}_0,{\bf L}_1$ of $\cQ$ are linkable in $B'$, and
    \item $|B_{1}|=(sk-s\ell-1)|A_1|+\ell$, where $A_1=A'\setminus V(\cQ)$ and $B_1=(B'\setminus V(\cQ))\cup L_0\cup L_1$.
\end{itemize}
\end{claim}

\begin{proof}
To prove the claim, we will discuss the $\ell$-path $\cQ$ in $\cH$ in two cases.\\
{\bf Case 1.} $q=|A\cap B'|\neq 0$.

Let 
$V_{0}=\{x_{1},x_{2},\ldots,x_{|V_{0}|}\}$. 
Recall that there are $sq$ vertex-disjoint $\ell$-paths $P_1,P_2,\ldots,P_{sq}$ in $B'$, each of which contains two linkable $\ell$-ends and has length at most $s_*$.
Let $\cP_0$ denote the family of these $sq$ $\ell$-paths and then $|V(\cP_0)|\le sq(\ell+s_*(k-\ell))$.
Next for each $x_i\in V_0$, we want to choose $\ell+s(k-\ell)-1$ vertices of $B$ ($B\subseteq B'$) together with $x_i$ forming an $\ell$-path of length $s$ such that these $|V_{0}|$ paths are pairwise vertex-disjoint, and also vertex-disjoint from the existing paths in $\cP_0$, 
and all these paths have linkable $\ell$-ends. 
Now suppose that we have found such $\ell$-paths for $x_1,x_2,\ldots,x_{i-1}$ with $i\leq |V_0|$ and we shall show that we can find another such path.
By the definition of $V_0$, we have $\deg (x_i,B)> \eps_1\binom{|B|}{k-1}$.
Let $\cG_{x_i}$ be the $(k-1)$-graph on $B$ such that $e'\in E(\cG_{x_i})$ if
\begin{itemize}
    \item $e'\cup \{x_i\}\in E(\cH)$,
    \item $e'$ does not contain any vertex from the existing paths,
    \item $e'$ does not contain any $\eps_1$-bad subset.
\end{itemize}
We derive a lower bound for $e(\cG_{x_i})$. 
Obviously, $A\setminus A'=(A\cap B')\cup V_0$ because $B\subseteq B'$. 
Thus, by Lemma~\ref{11}, we have $|V_0|\le 2\eps_2|B|$ and $q\le |A\setminus A'|\le \eps_2 |B|$.
Notice that 
$|V(\cP_0)|+(\ell+s(k-\ell)- 1)(i-1)<ss_*kq+sk|V_0|\leq 2s_*sk\eps_2|B|$
and consequently at most $2s_*sk\eps_2|B|\binom{|B|-1}{k-2}<2s_*sk^2\eps_2\binom{|B|}{k-1}$ $(k-1)$-sets of $B$ intersect the existing paths.
Recall that the number of $(k-1)$-sets in $B$ containing $\eps_1$-bad subset is at most
$2^{k-1}\eps_2{|B| \choose {k-1}}$ by Fact \ref{fact:no bad}. 
Thus, we derive that
\[e(\cG_{x_i})> \eps_1\binom{|B|}{k-1}-2s_*sk^2\eps_2\binom{|B|}{k-1}-2^{k-1}\eps_2\binom{|B|}{k-1}\geq \frac{\eps_1}{2}\binom{|B|}{k-1}\]
since $\eps_2\ll \eps_1$ and $|B|$ is sufficiently large. 
Since $\cG_{x_i}$ is a $(k-1)$-graph on $B$, by Lemma~\ref{turan number}, the upper bound of $\mathrm{ex}(|B|, P^{k-1}_{k})$ is $\frac{k+1}{2} \binom{|B|}{k-2}<\frac{\eps_1}{2}\binom{|B|}{k-1}$.
Then we can find a copy of $(k-1)$-uniform tight path $P^{k-1}_k$ in $\cG_{x_i}$, which together with $x_i$, forms a copy of $k$-uniform tight path $P'$ of length $k$ (so of order $2k-1$).
Now we can take a subpath of $P'$ of order $\ell+s(k-\ell)$, so that on the path, there are at least $\ell$ vertices both before and after $x_i$.
Finally, this tight path contains a ($k$-uniform) $\ell$-path of length $s$, denoted by $P'_i$, as a subgraph (and we keep the ordering of the vertices).
By construction, we know that all edges of $P'_i$ contain no $\eps_1$-bad subsets of $B$, which yields that both ends of $P'_i$ are linkable.
Thus, we can always find a desired $\ell$-path of length $s$ containing $x_i$.

Let $\cP_1$ denote the family of $|V_0|$ $\ell$-paths we just obtained.
Now we apply Claim~\ref{c2} repeatedly to connect the ends of two $\ell$-paths in $\cP_0$ and $\cP_1$ to a single $\ell$-path $\cP$ while avoiding the vertex set $U$ of all previously used vertices.
Repeating this manner, we would obtain an $\ell$-path $\cP$ of length at most  $s_*\cdot sq+s|V_0|+2s(|V_0|+sq-1)\leq ss_*q+3s^2(|V_0|+q)-2s\le (ss_*+3s^2)\eps_2|B|-2s$,
as $|V_0|+q=|A\setminus A'|\le \eps_2 |B|$ by Lemma~\ref{11}, 
which yields that $|V(\cP)|\le (ss_*+3s^2)k\eps_2|B|$.
Therefore, this process is possible as the number of vertices to avoid is at most 
$(ss_*+3s^2)k\eps_2|B|<\frac{n}{2s(k-\ell)}$.
Therefore, we obtain an $\ell$-path $\cP$ containing all vertices from $sq$ $\ell$-paths from the statement of the theorem and $V_0$, such that $|V(\cP)|\le (ss_*+3s^2)k\eps_2|B|$ and $\cP$ has two linkable ends.

Now we assume that 
\begin{equation*}\label{eqn:shet}
w:=(sk-s\ell-1)|A'\setminus V(\cP)|-|B'\setminus V(\cP)|=(sk-s\ell)|A'\setminus V(\cP)|-(n-|V(\cP)|).
\end{equation*}
As $q>0$, we have $|A'\setminus V(\cP)|\le |A|-1$ and thus $w\le |V(\cP)|\le (ss_*+3s^2)k\eps_2|B|$.
Moreover, when modulo $k-\ell$, we have
${w}\equiv {|V(\cP)|} \equiv \ell$,
(as $n\in (k-\ell)\mathbb N$) which implies that $\frac{w-\ell}{k-\ell}\in \mathbb N$.
Next we extend $\cP$ to an $\ell$-path $\cQ$ by applying Claim~\ref{claim:extend l-set} $\frac{w-\ell}{k-\ell}$ times and this is possible as the number of vertices to avoid is at most $|V(\cP)|+w\le (2ss_*+6s^2)k\eps_2|B|<\frac{n}{2s(k-\ell)}$.
Note that $\cQ$ also has two linkable $\ell$-ends, denoted by ${\bf L}_0, {\bf L}_1$ and $|V(\cQ)|=|V(\cP)|+w-\ell< (2ss_*+6s^2)k\eps_2|B|$.
Since $V(\cQ)\setminus V(\cP)$ contains $\frac{w-\ell}{k-\ell}$ vertices of $A'$ and $\frac{w-\ell}{k-\ell}((s-1)(k-\ell)-1)$ vertices of $B'$, we have 
\begin{eqnarray}\label{eqn:vxq}
\begin{aligned}
&(sk-s\ell-1)|A'\setminus V(\cQ)|-|B'\setminus V(\cQ)|\\
&=w-(sk-s\ell-1)\cdot \frac{w-\ell}{k-\ell}+\frac{w-\ell}{k-\ell}\cdot ((s-1)(k-\ell)-1)=\ell.
\end{aligned}
\end{eqnarray}

Let $A_1=A'\setminus V(\cQ)$ and $B_1=(B'\setminus V(\cQ))\cup L_0\cup L_1$. Then, $|B_1|=(sk-s\ell-1)|A_1|+\ell$ by \eqref{eqn:vxq}.
\\
{\bf Case 2.} $q=|A\cap B'|=0$.

This case is indeed treated similar to Case 1 except that $\cP_0=\emptyset$. 
Let $V_{0}=\{x_{1},x_{2},\ldots,x_{|V_{0}|}\}$ and suppose that there are $i$ $(i<|V_0|)$ $\ell$-paths of length $s$. 
Note that for every $x_i\in V_0$,
\begin{eqnarray*}\label{eqn:vx}
\begin{aligned}
\deg(x_i,B')&\ge \deg(x_i,B)-|B\setminus B'|\binom{|B|-1}{k-2}\\
&> \eps_1\binom{|B|}{k-1}-\eps_2|B|\binom{|B|-1}{k-2}\geq \frac{\eps_1}{2}\binom{|B'|}{k-1}
\end{aligned}
\end{eqnarray*}
by Lemma~\ref{11} and $\eps_2\ll \eps_1$. 
We then construct $\cG_{x_i}$ as in Case 1 and obtain $e(\cG_{x_i})\ge \frac{\eps_1}{4}\binom{|B'|}{k-1}$. 
Since $\cG_{x_i}$ is a $(k-1)$-graph on $B$, by Lemma~\ref{turan number}, the upper bound of $\mathrm{ex}(|B|, P^{k-1}_{k})$ is $\frac{k+1}{2} \binom{|B|}{k-2}<\frac{\eps_1}{4}\binom{|B'|}{k-1}$
as $|B|=|B'|+|B\setminus B'|\leq |B'|+\eps_2|B|$.
Then we can find a copy of $(k-1)$-uniform tight path $P^{k-1}_k$ in $\cG_{x_i}$.
Similar as in the previous case, we can choose $P^*_i$ as a $k$-uniform $\ell$-path containing $x_i$ such that both ends of $P^*_i$ are linkable.
Thus, there is a desired $\ell$-path of length $s$ containing $x_i$.
Next apply Claim \ref{c2} repeatedly to connect the ends of two $\ell$-paths to a single $\ell$-path $\cP$ while avoiding all previously used vertices.
This would yield an $\ell$-path $\cP$ of length $s|V_0| + 2s(|V_0|-1) = 3s|V_0| - 2s$, which is possible as the number of vertices to avoid is at most $3sk|V_0|\le 6sk\eps_2|B|<\frac{n}{2s(k-\ell)}$. 
Thus, we can build the desired path $\cP$ with linkable ends satisfying that $|V(\cP)|\le 6sk\eps_2|B|$.

Here $w$ is defined as in Case 1.
Thus, we have 
\begin{eqnarray*}\label{eqn:t2}
\begin{aligned}
w\le (sk-s\ell)|A'|-n+|V(\cP)|
\le sk\eps_2|B|+6sk\eps_2|B|= 7sk\eps_2|B|.
\end{aligned}
\end{eqnarray*}
Similar to the previous case, we have ${w}\equiv {|V(\cP)|} \equiv \ell \mod (k-\ell)$, which implies that $\frac{w-\ell}{k-\ell}\in \mathbb N$.
We extend $\cP$ to an $\ell$-path $\cQ$ by applying Claim~\ref{claim:extend l-set} $\frac{w-\ell}{k-\ell}$ times and this is possible as the number of vertices to avoid is at most $7sk\eps_2|B|<\frac{n}{2s(k-\ell)}$.
Note that $\cQ$ has two linkable $\ell$-ends and 
\begin{eqnarray*}
\begin{aligned}
|V(\cQ)|=|V(\cP)|+w-\ell\le 13sk\eps_2|B|.
\end{aligned}   
\end{eqnarray*}
It is easy to see that~\eqref{eqn:vxq} also holds and we define $A_1$ and $B_1$ in the same way as in Case 1.

If $V_0= \emptyset$, then by the definition of $A'$, we can arbitrarily choose a vertex $a'\in A'$ forming an $\ell$-path $\cP$ of length $s$ with two linkable $\ell$-ends such that all other vertices are from $B'$ and $a'$ is in the intersection of these $s$ edges. 
Thus, we have $|V(\cP)|=\ell+s(k-\ell)$.
Define $w$ as in Case 1. 
We extend $\cP$ to an $\ell$-path $\cQ$ by applying Claim~\ref{claim:extend l-set} $\frac{w-\ell}{k-\ell}$ times and this is possible as the number of vertices to avoid is at most $7sk\eps_2|B|<\frac{n}{2s(k-\ell)}$.
Then $|V(\cQ)|=\ell+s(k-\ell)+w-\ell\le 8sk\eps_2|B|$.
The rest is the same as in the previous case.
\end{proof}

After constructing the short $\ell$-path $\cQ$, we can derive the following properties. 
\begin{proposition}\label{qpro}
Define $A_1,B_1$ 
and ${\bf L}_0,{\bf L}_1$ 
as in Claim \ref{cover v0}.
Then we have the following properties:
\begin{enumerate}[label=(\alph*)]
    \item $|B_1|\ge (1-\eps_1)|B|$,
    \item for every $a\in A_1$, $\overline{\deg}(a,B_1)<3\eps_1\binom{|B_1|}{k-1}$
    and for every $b\in B_1$, $\overline{\deg}(b,A_1B^{k-1}_1)\le 3k\eps_1\binom{|B_1|}{k-1}$,
    \item 
for each $i\in [s-1]\cup \{0\}$ and $j\in \{0,1\}$, 
$\overline{\deg}(L_j^i,A_1B^{k-1}_1)\le (2s+1)k\eps_1\binom{|B_1|}{k-\ell_i}$, where $L_j^i$ denotes the $\ell_i$-set of ${\bf L}_j$, with $\ell_0=0$ and $\ell_i=i\ell-(i-1)k$ for $i\in[s-1]$.
\end{enumerate}
\end{proposition}
\begin{proof}
The proofs of (a)-(b) are omitted here, since they are identical to the proofs presented in~\cite[Claim 3.9]{HZ15}. 

For (c), by Proposition~\ref{c1}, every
$\varepsilon_{1}$-good $\ell'$-set $L'\subseteq B_1\subseteq B'$ satisfies $\overline{\deg}(L',A'(B')^{k-1})\le 2sk\varepsilon_{1}|A'|\binom{|B'|-\ell'}{k-1-\ell'}$.
Thus, we have
\[
\overline{\deg}(L',A_1B^{k-1}_1)\le \overline{\deg}(L',A'(B')^{k-1}) \le 2sk\varepsilon_{1}|A'|\binom{|B'|-\ell'}{k-1-\ell'}\le (2s+1)k\eps_1\binom{|B_1|}{k-\ell'}
\]
as $|B'|\le |B_1|+|V(\cQ)|\le (1+\eps_1)|B_1|$. 
For each $i\in [s-1]\cup \{0\}$ and $j\in \{0,1\}$, since $\ell_i$-set $L_j^i\subseteq B_1$ is $\eps_1$-good, we have $\overline{\deg}(L_j^i,A_1B^{k-1}_1)\le (2s+1)k\eps_1\binom{|B_1|}{k-\ell_i}.$
\end{proof}

\subsubsection{Completing the Hamilton $\ell$-cycle} 
Now we move to the second phase of the proof.
Note that there exists an $\ell$-path $\cQ$ with two linkable $\ell$-ends ${\bf L}_0,{\bf L}_1$ by Claim~\ref{cover v0}.
To complete the Hamilton $\ell$-cycle in Theorem~\ref{thm:main result}, it suffices to apply the following result with $X=A_1,Y=B_1,\rho=(2s+1)k\eps_1$,  
and ${\bf L}'_0,{\bf L}'_1$,
where ${\bf L}'_j$ is the ordered $\ell$-set obtained by reversing the ordering of ${\bf L}_j$ for $j\in\{0,1\}$.
\begin{claim}\label{clm:partition}
Let $k/2< \ell < k, 0 < \rho\ll 1$ and $n$ be sufficiently large. 
Suppose that $\cH$ is a $k$-graph with a partition $V(\cH) = X \cup Y$ and the following properties:
\begin{enumerate}[label=(\arabic{enumi})]
    \item $|Y| = (sk - s\ell - 1)|X| + \ell$,
    \item for every vertex $v \in X$, $\overline{\deg}(v, Y) \leq \rho\binom{|Y|}{k-1}$ and for every vertex $v \in Y$, $\overline{\deg}(v, XY^{k-1}) \leq \rho\binom{|Y|}{k-1}$,
    \item for each $i\in [s-1]\cup \{0\}$ and $j\in \{0,1\}$,
$\overline{\deg}(L_j^i,XY^{k-1})\le \rho \binom{|B_1|}{k-\ell_i}$, where $L_j^i$ denotes the $\ell_i$-set of ${\bf L}_j$, with $\ell_0=0$ and $\ell_i=i\ell-(i-1)k$ for $i\in[s-1]$.
\end{enumerate}
Then $\cH$ contains a Hamilton $\ell$-path with two $\ell$-ends ${\bf L}'_0$ and ${\bf L}'_1$, where ${\bf L}'_j$ is the ordered $\ell$-set obtained by reversing the ordering of ${\bf L}_j$ for $j\in\{0,1\}$.
\end{claim}

In order to prove Claim~\ref{clm:partition}, we apply the Lov\'asz Local Lemma (see e.g.~\cite{AS2000})
and a result of Lu and Sz\'ekely~\cite[Theorem 1]{LS2007}.

Let $A_1,A_2,\ldots,A_n$ be events in a probability space $\Omega$.
Given an $n$-vertex graph $G$, we call it a \emph{dependency graph} of the events $A_i$'s if $A_i$ is mutually independent of all $A_j$'s with $ij\notin E(G)$.
More generally, a \emph{negative dependency graph} for $A_1,A_2,\ldots,A_n$ is an  $n$-vertex simple graph $G$ satisfying 
$\mathbb{P}[A_i\mid \wedge_{j\in S}\overline{A_j}]\leq \mathbb{P}[A_i]$
for any index $i$ and any subset $S\subseteq \{j\mid ij\notin E(G)\}$, if the conditional probability $\mathbb{P}[\wedge_{j\in S}\overline{A_j}]>0$.

\begin{lemma}\emph{(Lov\'asz Local Lemma)}\label{lem:lll}
\emph{\cite{AS2000}}
 For each $1\leq i\leq n$, suppose that the event $A_i$ satisfies $\mathbb{P}(A_i)\leq p$ and a negative dependency graph $G$ is associated with these events.
 Let $d$ be an upper bound for the degrees in $G$.
 If $ep(d+1)<1$, then $\mathbb{P}[\wedge_{i=1}^n\overline{A_i}]>0$.
\end{lemma}



Given two finite sets $U$ and $W$ with $|U|\le|W|$, let $I(U,W)$ denote the probability space of all injections from $U$ to $W$ equipped with
a uniform distribution.
Note that every injection from $U$ to $W$ can
be viewed as a saturated matching of complete bipartite graph with partite sets $U$ and $W$.
We define a matching to be a triple $(S,T,f)$ satisfying $S\subseteq U,T\subseteq W$ and the map $f:S\to T$ is a bijection.
Denote the set of all such matchings by $M(U,W)$, and  $I(U,W)\subseteq M(U,W)$.
Given a matching $(S,T,f)$, we define the event $A_{S,T,f}$ as 
\[A_{S,T,f}=\{\sigma\in I(U,W)\mid \sigma(i)=f(i),\forall i\in S\}.\]
An event $A\in I(U,W)$ is called to be \emph{canonical} if $A=A_{S,T,f}$ for a matching $(S,T,f)$.
Two matchings $(S_1,T_1,f_1)$ and $(S_2,T_2,f_2)$ are said to \emph{conflict} each other if either there is $v\in S_1\cap S_2$ such that $f_1(v)\neq f_2(v)$, or there is $v\in T_1\cap T_2$ such that $f^{-1}_1(v)\neq f^{-1}_2(v)$.

Lu and Székely~\cite[Theorem 1]{LS2007} established a sufficient condition for negative dependency graphs in the space of random injections.
\begin{lemma}\label{lem:negative}\emph{\cite{LS2007}}
Let $A_1, A_2, \dots, A_n$ be canonical events in $I(U, W)$. Let $G$ be the graph on $[n]$ defined as 
\[
E(G) = \{ij \mid A_i\textrm{ and }A_j \textrm{ conflict}\}.
\]
Then $G$ is a negative dependency graph for the events $A_1, \dots, A_n$.    
\end{lemma}
Now we are ready to prove Claim~\ref{clm:partition}.
Our goal is to show that the desired Hamilton $\ell$-path can be found in $\cH$ by Lov\'asz Local Lemma.
The following proof draws idea heavily from the proof of~\cite[Theorem 2]{LS2007}. 

\begin{proof}[Proof of Claim~\ref{clm:partition}]
Let $\cH_1$ be a Hamilton ($k$-uniform) $\ell$-path of length $st$ and $\cH_2$ with $E(\cH_2):=\overline{E}_{\cH}(XY^{k-1})$. 
Now we aim to prove that there exists an injection $f:V(\cH_1)\to V(\cH)$ which is extended to $E(\cH_1)$ in the natural way, 
such that the images of edges of $\cH_1$ and $\cH_2$ are edge-disjoint.
Let $t=|X|$ and we assume $X=\{x_1,x_2,\ldots,x_{t}\}$.
Suppose that ${\bf L}_0=(y_{st(k-\ell)+\ell},\ldots,y_{st(k-\ell)+1})$ and ${\bf L}_1=(y_1,\ldots,y_\ell)$.
Let ${\bf L}'_0=(y_{st(k-\ell)+1},\ldots,y_{st(k-\ell)+\ell})$ and ${\bf L}'_1=(y_\ell,\ldots,y_1)$.
Define that $\ell_0=0$ and $\ell_i=i\ell-(i-1)k$ for $i\in [s-1]$.
For $i\in [s-1]\cup \{0\}$, let $L^i_0=\{v_{(st-i)(k-\ell)+\ell},v_{(st-i)(k-\ell)+\ell-1},\ldots,v_{st(k-\ell)+1}\}$ and $L^i_1=\{v_{(i-1)(k-\ell)+1},v_{(i-1)(k-\ell)+2},\ldots,v_\ell\}$ denote the $\ell_i$-sets of ${\bf L}_0$ and ${\bf L}_1$, respectively. 
In addition, we let $L_0:=L^1_0$ and $L_1:=L^1_1$.
By (3) in Claim~\ref{clm:partition}, we have $\overline{\deg}(L^i_1,XY^{k-1})\le \rho\binom{|Y|}{k-\ell_{i}}$ with $i\in [s-1]\cup \{0\}$.
Note that $\overline{\deg}(L^i_1,XY^{k-1})=\sum_{v\in X}\overline{\deg}(L^i_1\cup \{v\},XY^{k-1})\le \rho\binom{|Y|}{k-\ell_i}$.
Choose a vertex $x_1\in X$ randomly.
Let $\rho':=2s^2\rho\binom{|Y|}{k-\ell_i-1}$.
By Markov's inequality, we derive 
\begin{eqnarray}\label{eqn:edge pro}
 \mathbb{P}[\overline{\deg}(L^i_1\cup \{x_1\},XY^{k-1})\ge \rho']\le 
 \frac{\rho}{|X|\rho'}\binom{|Y|}{k-\ell_i}\le \frac{|Y|}{2s^2(k-\ell_i)|X|}\le \frac{1}{2s}
 \end{eqnarray}
where the last inequality follows from $|Y| = (sk- s\ell - 1)|X| + \ell$. 
Thus, by the union bound, we have that $\overline{\deg}(L^i_1\cup \{x_1\},XY^{k-1})\le 2s^2\rho\binom{|Y|}{k-\ell_i-1}$ for all $i\in [s-1]\cup \{0\}$ with probability at least $1/2$, so that we can fix such a vertex $x_1\in X$.
Similarly, for $L^i_0$, there exists $x_t\in X\setminus\{x_1\}$ such that 
$\overline{\deg}(L^i_0\cup \{x_t\},XY^{k-1})\le  2s^2\rho\binom{|Y|}{k-\ell_i-1}$  for all $i\in [s-1]\cup \{0\}$.
Let $V(\cH_1)=\{v_1,\dots, v_{\ell+st(k-\ell)}\}$ and 
$Y_1=\{v_j\mid j\in \{\ell+1,\ldots,st(k-\ell)\}\setminus \bigcup_{i=1}^t\{(i-1)(sk-s\ell)+k\}\}.$
Given any labeling of $X\setminus\{x_1,x_t\}=\{x_2, x_3,\dots, x_{t-1}\}$, we consider a random injection $f'$ from $Y_1\subseteq V(\cH_1)$ to $Y\setminus (L_0\cup L_1)$.
To extend $f'$ to an injection $f$ from $V(\cH_1)$ to $V(\cH)$, we let $x_i=f(v_{(i-1)(sk-s\ell)+k})\in X$ for $i\in [t]$ and let $y_j=f(v_j)\in Y$ for
$j\in [\ell] \cup \{st(k-\ell)+1,\ldots,st(k-\ell)+\ell\}$.
Note that $L_0=\{y_{st(k-\ell)+\ell},\ldots, y_{st(k-\ell)+1}\}$ and $L_1=\{y_1,\ldots,y_\ell\}$.
Denote the first $s$ and the last $s$ edges of $\cH_1$ by $e^*_1,\dots,e^*_s$ and $e^*_{st-s+1},\dots,e^*_{st}$, respectively.
For each $i\in[s]\cup\{st-s+1,\ldots,st\}$, let $e_i:=f(e_i^*)$ be its image under $f$ and let $B_{e_i}$ be the event that 
$e_i\in E(\cH)$, and define $B_e=\bigcup_i B_{e_i}$.
By the union bound and~\eqref{eqn:edge pro}, we have $\mathbb{P}[B_e]\geq 1-4s^3\rho$.
From now on we condition on the event $B_e$.

Our probability space is $I(U,W)$ with $U=Y_1$ and $W=Y\setminus (L_0\cup L_1)$.
Consider two edges $F_1\in E(\cH_1)\setminus\{e^*_j\mid j\in [s]\cup\{st-s+1,\ldots,st\}\}$ and $F_2\in E(\cH_2)$ as well as a bijection $\phi:F_1\to F_2$.
 Define (canonical) events $A_{F_1,F_2,\phi}=\{\sigma\in I(U,W)\mid\sigma(i)=\phi(i),\forall i\in V(F_1)\}$ 
 as our bad events.
Note that $\mathbb{P}[A_{F_1,F_2,\phi}]=\frac{1}{|X|\binom{|Y|}{k-1}(k-1)!}$.
 Thus, for the conditional probability $\mathbb{P}[A_{F_1,F_2,\phi}\mid B_e]$, we have 
 \[
 \mathbb{P}[A_{F_1,F_2,\phi}\mid B_e]\leq \frac{\mathbb{P}[A_{F_1,F_2,\phi}]}{\mathbb{P}(B_e)}\leq \frac{1}{1-4s^3\rho}\cdot \frac{1}{|X|\binom{|Y|}{k-1}(k-1)!}<\frac{1+8s^3\rho}{|X|\binom{|Y|}{k-1}(k-1)!} =: p_0,
 \]
 where the last inequality holds as $\rho\ll1$.

Let $G$ be the conflict graph defined on the family of all events $A_{F_1,F_2,\phi}$, that is, two vertices $A_{F_1,F_2,\phi}$ and $A_{F_1',F_2',\phi'}$ are connected by an edge if and only if they conflict with each other.
By definition, an event $A_{F_1,F_2,\phi}$ conflicts another event $A_{F'_1,F'_2,\phi'}$ if and only if 
\begin{itemize}
    \item $F_1\cap F'_1=\emptyset$ and $F_2\cap F'_2\neq \emptyset$, or
    \item $F_1\cap F'_1\neq \emptyset$ and $\phi(F_1\cap F'_1)\neq \phi'(F_1\cap F'_1)$.
\end{itemize}
For $i=1,2$, suppose that $\cH_i$ has $m_i$ edges and every edge in $\cH_i$ intersects at most $d_i$ other edges of $\cH_i$. 
Given $x\in X$ and $y\in Y$, by (2) in Claim~\ref{clm:partition}, we obtain that an event $A_{F_1,F_2,\phi}$ has at most 
\[k!(d_2+1)m_1=k!\left(\overline{\deg}(x,Y)+\overline{\deg}(y,XY^{k-1})+1\right)\cdot|X|\leq 3k!\rho|X|\binom{|Y|}{k-1}\]
conflicts of the first type, and at most 
\[k!(d_1+1)m_2=(s+1)k!\overline{e}(XY^{k-1})=(s+1)k!\sum_{v\in X}\overline{\deg}(v,Y)\leq (s+1)k!\rho|X|\binom{|Y|}{k-1}\] conflicts of the second type.
Therefore, the maximal degree $d$ of $G$ is at most
$(s+4)k!\rho|X|\binom{|Y|}{k-1}$.
So, by Lemma~\ref{lem:negative}, $G$ is a negative dependency graph for the events $A_{F_1,F_2,\phi}$.
Since $\rho\ll 1$, we derive
\[e(d+1)p_0<e\left((s+4)k!\rho|X|\binom{|Y|}{k-1}+1\right)\cdot \frac{1+8s^3\rho}{|X|\binom{|Y|}{k-1}(k-1)!}<1. \]
Thus, by Lemma~\ref{lem:lll}, with positive probability, all bad events $A_{F_1,F_2,\phi}$ condition on $B_e$
do not occur.
That is, there exists an injection $f:V(\cH_1)\to V(\cH)$ which is extended to $E(\cH_1)$ in the natural way, 
such that the images of edges of $\cH_1$ and $\cH_2$ are edge-disjoint, yielding that all edges of $\cH_1$ are mapped to edges of $\cH$.
So we obtain a Hamilton ($k$-uniform) $\ell$-path of length $st$ of $\cH$ with ${\bf L}'_0$ and ${\bf L}'_1$ as ends.
\end{proof}

The rest of this section is devoted to showing that there are $sq$ vertex-disjoint $\ell$-paths $P_1,P_2,\dots,P_{sq}$ with linkable $\ell$-ends in $B'$.

\subsection{Finding $sq$ vertex-disjoint $\ell$-paths under minimum co-degree conditions}\label{subsec:sq}
For any vertex $b \in B'$, by the definition of $B'$, we have
\begin{eqnarray}\label{b1}
\begin{aligned}
   \deg(b, B')&\leq \deg(b, B)+|B'\setminus B|\binom{|B'|-1}{k-2}\\
   &\leq \varepsilon_{1}\binom{|B|}{k-1}+\varepsilon_{2}|B|\binom{|B'|-1}{k-2}<2\varepsilon_{1}\binom{|B|}{k-1}.
\end{aligned}
 \end{eqnarray}

For the case $k/2<\ell<3k/4$, we firstly use the exact minimum co-degree condition~\eqref{eqn:conj value} to find the following $sq$ vertex-disjoint edges in $B'$ and then extend them to $sq$ $\ell$-paths $P_1,P_2,\ldots,P_{sq}$, each of which has length at most seven and contains two linkable $\ell$-ends.

 
\begin{claim}\label{clmu}
Suppose $k\ge 3, k/2 <\ell<3k/4$ such that $(k-\ell)\nmid k$ and $\delta_{k-1}(\cH)\geq \frac{n}{s(k-\ell)}$.
    Let $|A\cap B'|=q>0$. Then there is 
    a matching $\cP^*$ of size $sq$ in $B'$, denoted by $e_1,e_2,\ldots,e_{sq}$, such that for $i\in [sq]$ there exists $u_i\in e_i$ such that $e_i\setminus \{u_i\}$ has no $\eps_1^2/3$-bad subset.
\end{claim}

\begin{proof}
Since $A\cap B'\neq \emptyset$, by Lemma~\ref{lem1}, we have $B\subseteq B'$. 
Given a $(k-1)$-set $K\subseteq B$, we can find a neighbor of $K$ in $B'$ by the degree condition. 
Since $|B'|=|B|+|B'\setminus B|=|B|+|A\cap B'|= \left\lfloor \left(1-\frac{1}{s(k - \ell)}\right) n \right\rfloor+q$, by~\eqref{eqn:conj value}, we have $\deg(K,B')\geq q$. 
Denote by $\cP^{**}$ the family of edges in $B'$ such that for each $e\in \cP^{**}$, there exists $u\in e$ such that $e\setminus \{u\}$ has no $\eps_1^2/3$-bad subset.
By Fact~\ref{fact:no bad}, there are at least 
$(1-3\cdot2^{k-1}\eps_2)\binom{|B|}{k-1}$ $(k-1)$-sets in $B$ containing no $\eps_1^2/3$-bad subset.
So we derive a lower bound for $|\cP^{**}|$, that is, 
\[|\cP^{**}|> (1-3\cdot2^{k-1}\eps_2)\binom{|B|}{k-1}\frac{q}{k},\]
in which we divide by $k$ because each edge of $\cP^{**}$ is counted at most $k$ times.

Next, we prove that $\cP^{**}$ contains a matching of size $sq$.
Suppose instead, a maximum matching has $i$ ($i<sq$) edges. By \eqref{b1}, there are at most $sqk\cdot 2\eps_{1}\binom{|B|}{k-1}$ edges of $B'$ intersecting these $i$ edges in the matching. Thus, the number of edges in $\cP^{**}$ that are disjoint from these $i$ edges is at least 
\[(1-3\cdot2^{k-1}\eps_2)\binom{|B|}{k-1}\frac{q}{k}-2sk\eps_{1}q\binom{|B|}{k-1}\ge \left(1-3\cdot2^ksk^2\eps_{1}\right)\binom{|B|}{k-1}\frac{q}{k}>0\]
as $\eps_2\ll\eps_{1}\ll \frac{1}{sk^k}$. That is, we can extend the current matching to a larger one, a contradiction.        
\end{proof}

We next extend each edge in Claim \ref{clmu} to an $\ell$-path with two linkable $\ell$-ends.
For $k/2< \ell<3k/4$, we have $3\ell-2k<k-\ell<2\ell-k<2k-2\ell$ and $s\le 4$.
Then when we extend an $\ell$-path from an $\ell$-end $(v_1,\dots, v_\ell)$, (recalling the definition of linkable) there are at most three relevant sets: the $\ell$-set $\{v_1,\dots, v_\ell\}$, the $(2\ell-k)$-set $\{v_{k-\ell+1},\dots, v_\ell\}$ and (if $\ell > 2k/3$) the $(3\ell-2k)$-set $\{v_{2k-2\ell+1},\dots, v_\ell\}$. 
As $3\ell-2k < k-\ell$, it is easy to choose the $(3\ell-2k)$-set to be $\eps_1$-good.
As $k-\ell<2\ell-k<2k-2\ell$, the number of ways of extending an $\eps_1$-good $(3\ell-2k)$-set to an $\eps_1$-bad $(2\ell-k)$-set is very small, which implies that we can always choose the $(2\ell-k)$-set to be $\eps_1$-good.
Then the main concern is the $\ell$-set -- however, if the $\ell$-set is $\eps_1$-bad, then there are many extensions of the path. 
Extending the path three times gives us enough number of extensions so that we can choose a good extension, that is, with a linkable $\ell$-end (see Figure~\ref{fig:1}).

\begin{figure}[ht]
\begin{tikzpicture}

	\coordinate (v01) at (1,0);
	\coordinate (v02) at (2,0);
	
	\coordinate (v1) at (3,0);
	\coordinate (v2) at (4,0);
	\coordinate (v3) at (5,0);
	\coordinate (v4) at (6,0);

	\coordinate (v5) at (7,0);

	\coordinate (v6) at (8,0);
	\coordinate (v7) at (9,0);
	\coordinate (v8) at (10,0);
	\coordinate (v9) at (11,0);
	
	\coordinate (v10) at (12,0);
	\coordinate (v11) at (13,0);

	\begin{pgfonlayer}{front}
		
		\foreach \i in {v1, v2, v3, v4, v6, v5, v7, v8, v9, v10, v11}
			\fill  (\i) circle (2pt);

		\node at (v02) {$\cdots$};

	\end{pgfonlayer}
	
	\begin{pgfonlayer}{background}
	\draw[thick, decorate, decoration={brace, amplitude=8pt}]  (7.2, -1.35) -- (2.8, -1.35) node[midway, below=6pt, font=\scriptsize] {$L_0$};
	\draw[thick, decorate, decoration={brace, amplitude=5pt}]  (7.2, -0.9) -- (4.8, -0.9) node[pos=0.5, below=4pt, font=\scriptsize] {$S_0$};
	\draw[thick, decorate, decoration={brace}]  (7.2, -0.5) -- (6.8, -0.5) node[midway, below=1pt, font=\scriptsize] {$T_0$};
	\draw[thick, decorate, decoration={brace, amplitude=4pt}]  (9.2, -0.9) -- (7.6, -0.9) node[midway, below=4pt, font=\scriptsize] {$W_0$};
	\draw[thick, decorate, decoration={brace, amplitude=4pt}]  (11.2, -0.9) -- (9.6, -0.9) node[midway, below=4pt, font=\scriptsize] {$W_1$};
	\end{pgfonlayer}
	
	
	\qedgeSeven[offset=12pt, line width=1.5pt,  fill=blue!20!white, fill opacity=0.3, draw=blue!70!black]{(v11)}{(v10)}{(v9)}{(v8)}{(v7)}{(v6)}{(v5)};
	\qedgeSeven[offset=12pt, line width=1.5pt, fill=red!20!white, fill opacity=0.7, draw=red!70!black]{(v9)}{(v8)}{(v7)}{(v6)}{(v5)}{(v4)}{(v3)};
	\qedgeSeven[offset=12pt, line width=1.5pt, fill=green!20!white, fill opacity=0.7, draw=green!70!black]{(v7)}{(v6)}{(v5)}{(v4)}{(v3)}{(v2)}{(v1)};

\begin{scope}
		\def\greyoffset{12pt}
		\def\lowerY{-0.45} 
		\def\highY{0.45}
		\def\lowerX{7.4}
		\def\leftX{1.3}
	
		\fill[lightgray!40, opacity=0.8, rounded corners=\greyoffset] 
   			 (\leftX, \highY) -- (\lowerX, \highY) -- (\lowerX, 0) -- (\lowerX, \lowerY) -- (\leftX, \lowerY) ;
		\draw[line width=1.5pt, gray!70!black, rounded corners=\greyoffset] 
			(\leftX, \highY) -- (\lowerX, \highY) -- (\lowerX, 0);
		\draw[line width=1.5pt, gray!70!black, rounded corners=\greyoffset] 
			(\lowerX, 0) -- (\lowerX, \lowerY) -- (\leftX, \lowerY);
\end{scope}	
\end{tikzpicture}
\caption{Extensions from an $\ell$-end $L_0$ and the relevant sets for $k=7, \ell=5$}
\label{fig:1}
\end{figure}
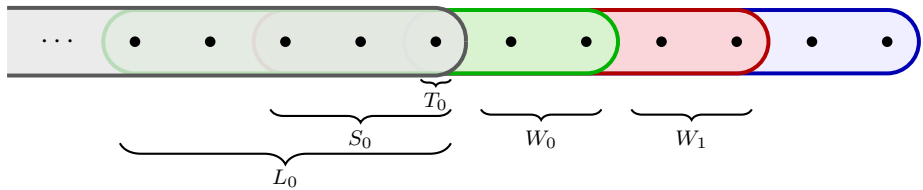

\begin{lemma}
[Disjoint paths in $B'$ for $k/2<\ell<3k/4$]
\label{lem:sq paths}
Suppose $k\ge 3,k/2<\ell<3k/4$ such that $(k-\ell)\nmid k$ and $\delta_{k-1}(\cH)\geq \frac{n}{s(k-\ell)}$.
Let $|A\cap B'|=q>0$. Then there exist $sq$ vertex-disjoint paths in $B'$, each of which has length at most seven and contains two linkable $\ell$-ends.
\end{lemma}
\begin{proof}
Since $|A\cap B'|=q>0$, by Lemma~\ref{lem1}, we have $B\subseteq B'$ and by Claim~\ref{clmu}, there exist vertex-disjoint edges $e_1,e_2,\ldots,e_{sq}$ in $B'$ such that for $i\in [sq]$ there exists $u_i\in e_i$ such that $e_i\setminus \{u_i\}$ has no $\eps_1^2/3$-bad subset.
We will show that $e_1$ can be extended to an $\ell$-path of length at most seven with two linkable $\ell$-ends while avoiding \textit{any} given vertex set $U_1\subseteq B'$ of size at most $skq\le sk \eps_2 |B|$. Then the claim easily follows from a greedy process for $e_2,\ldots,e_{sq}$ (while the size of the set to avoid gets larger but no larger than $7skq$).

Suppose $e_1=\{v_1,v_2,\ldots,v_k\}$ and  $u_1=v_{\lceil\frac{k}{2}\rceil}$.
Suppose that ${\bf L}_0=(v_{k-\ell+1},v_{k-\ell+2},\ldots,v_k)$ and ${\bf L}_1=(v_1,v_2,\ldots,v_\ell)$ are $\ell$-tuples of $e_1$, and we write $L_0$ and $L_1$ for the underlying $\ell$-sets of ${\bf L}_0, {\bf L}_1$, respectively. 
Let us extend $e_1$ starting from ${\bf L}_0$.
Define $S_0:=\{v_{2k-2\ell+1},v_{2k-2\ell+2},\ldots,v_k\}$ and $T_0:=\{v_{3k-3\ell+1},v_{3k-3\ell+2},\ldots,v_k\}$ (if $\ell < 2k/3$, then $T_0:=\emptyset$).
Since $2\ell-k<k/2$ and $3\ell-2k<k/2$, we have $u_1\notin S_0$ and $u_1\notin T_0$, implying that
$S_0$ and $T_0$ are $\eps_1^2/3$-good.
Assume that $L_0$ is $\eps_1$-bad (otherwise we are done with this side),  we have $\deg(L_0,B)> \eps_1\binom{|B|}{k-\ell}$.
Let $\cN_{L_0}$ be a family of $(k-\ell)$-sets contained in $B\setminus U_1$ such that $W_0\in \cN_{L_0}$ if and only if 
\begin{enumerate}[label=\roman*)]
    \item $L_0\cup W_0\in E(\cH)$,
    \item $W_0$ contains no $\eps_1^2/3$-bad subset, 
    \item $T_0\cup W_0$ is $\eps_1$-good.
\end{enumerate}
For the number of $W_0$ satisfying ii) but violating iii), we claim that the number of $(k-\ell)$-sets $W_0\subseteq B$ such that ii) holds but $T_0\cup W_0$ is $\eps_1$-bad, namely, $\deg(T_0\cup W_0,B)> \eps_1\binom{|B|}{k-(2\ell-k)}$, is at most $\frac{\eps_1}2\binom{|B|}{k-\ell}$.
Indeed, if $\ell < 2k/3$, then $T_0=\emptyset$ and every $W_0$ satisfying ii) also satisfies iii).
Otherwise, $\ell > 2k/3$, that is, $3\ell - 2k >0$ and suppose the claim fails to hold.
As $T_0$ is $\eps_1^2/3$-good, we have $\deg(T_0, B)\leq \frac{\eps_1^2}{3}\binom{|B|}{k-(3\ell-2k)}$ but our assumption implies $\deg(T_0,B)> 
\frac{\eps_1}2\binom{|B|}{k-\ell}\cdot \eps_1\binom{|B|}{k-(2\ell-k)}/\binom{3k-3\ell}{k-\ell}> \frac{\eps_1^2}{3}\binom{|B|}{k-(3\ell-2k)}$, a contradiction.  
By Fact~\ref{fact:no bad}, we derive a lower bound for $|\cN_{L_0}|$, that is,
\[
|\cN_{L_0}|> {\eps_1}\binom{|B|}{k-\ell} - |U_1|\binom{|B|}{k-\ell-1}-3\cdot2^{k-\ell}\eps_2\binom{|B|}{k-\ell}-\frac{\eps_1}{2}\binom{|B|}{k-\ell}>\frac{\eps_1}{4}\binom{|B|}{k-\ell}.
\]
Now we claim that we may assume 
\begin{enumerate}
    \item[iv)] for every $W_0\in \cN_{L_0}$, $S_0\cup W_0$ is $\eps_1$-bad.
\end{enumerate}
Indeed, fix any $W_0=\{v_{k+1},\ldots,v_{2k-\ell}\}\in \cN_{L_0}$.
Thus, $L_0\cup W_0\in E(\mathcal H)$ and $W_0$ contains no $\eps_1^2/3$-bad subset.
If the $\ell$-tuple $(v_{2k-2\ell+1},\dots, v_{2k-\ell})$ is linkable, then we are done with one end of the $\ell$-path by adding $W_0$ to the path $e_1$.
Otherwise, as both $W_0$ and $T_0\cup W_0$ are $\eps_1$-good, we infer that the $\ell$-set $L'_0=\{v_{2k-2\ell+1},\dots, v_{2k-\ell}\}$ is $\eps_1$-bad.

Now fix any $W_0=\{v_{k+1},\ldots,v_{2k-\ell}\}\in \cN_{L_0}$ and let $e'_1=L_0\cup W_0$.
Let $L'_0=S_0\cup W_0=\{v_{2k-2\ell+1},\dots, v_{2k-\ell}\}$, 
$S'_0=T_0\cup W_0=\{v_{3k-3\ell+1},\ldots,v_{2k-\ell}\}$ and $T'_0=\{v_{4k-4\ell+1},\ldots,v_{2k-\ell}\}$ (again $T_0'=\emptyset$ if $\ell < 2k/3$).
Let $\cN_{L'_0}$ be a family of $(k-\ell)$-sets in $ B\setminus U_1$ such that $W_1\in \cN_{L'_0}$ if and only if 
\begin{enumerate}[label=\alph*)]
    \item $L'_0\cup W_1\in E(\cH)$,
    \item $W_1$ contains no $\eps_1^2/3$-bad subset,
    \item $T'_0\cup W_1$ is $\eps_1$-good.
\end{enumerate}
For the number of $W_1$ satisfying b) but violating c), we claim that the number of $(k-\ell)$-sets $W_1\subseteq B$ such that b) holds but $T'_0\cup W_1$ is $\eps_1$-bad, namely, $\deg(T'_0\cup W_1,B)> \eps_1\binom{|B|}{k-(2\ell-k)}$, is at most $\frac{\eps_1}2\binom{|B|}{k-\ell}$.
Indeed, if $\ell < 2k/3$, then $T'_0=\emptyset$ and every $W_1$ satisfying b) also satisfies c).
Otherwise, $\ell > 2k/3$, that is, $3\ell - 2k >0$ and suppose the claim fails to hold.
As $T'_0\subseteq W_0$, $T'_0$ is $\eps_1^2/3$-good, that is, $\deg(T'_0, B)\leq \frac{\eps_1^2}{3}\binom{|B|}{k-(3\ell-2k)}$.  
But our assumption implies $\deg(T'_0,B)> 
\frac{\eps_1}2\binom{|B|}{k-\ell}\cdot \eps_1\binom{|B|}{k-(2\ell-k)}/\binom{3k-3\ell}{k-\ell}> \frac{\eps_1^2}{3}\binom{|B|}{k-(3\ell-2k)}$, a contradiction. 
By Fact~\ref{fact:no bad}, we derive that
\[
|\cN_{L'_0}|> {\eps_1}\binom{|B|}{k-\ell} - |U_1|\binom{|B|}{k-\ell-1}-3\cdot2^{k-\ell}\eps_2\binom{|B|}{k-\ell}-\frac{\eps_1}{2}\binom{|B|}{k-\ell}>\frac{\eps_1}{4}\binom{|B|}{k-\ell}.
\]
Now we claim that we may assume 
\begin{enumerate}
    \item[d)] for every $W_1\in \cN_{L'_0}$, $S'_0\cup W_1$ is $\eps_1$-bad.
\end{enumerate}
In fact, 
fix any $W_1=\{v_{2k-\ell+1},\ldots,v_{3k-2\ell}\}\in \cN_{L'_0}$.
That is, $L'_0\cup W_1\in E(\cH)$ and $W_1$ contains no $\eps_1^2/3$-bad subset.
If the $\ell$-tuple $(v_{3k-3\ell+1},\dots, v_{3k-2\ell})$ is linkable, then we are done with one end of the $\ell$-path by adding $W_0\cup W_1$ to the path $e_1$.
Otherwise, as both $W_1$ and $T'_0\cup W_1$ are $\eps_1$-good, we infer that the $\ell$-set $L''_0=\{v_{3k-3\ell+1},\dots, v_{3k-2\ell}\}$ is $\eps_1$-bad.

Therefore, given any choice of $W_0\cup W_1$, we may extend the path again, and by our assumption that the $\ell$-set $L''_0$ is $\eps_1$-bad, there are at least $\eps_1\binom{|B|}{k-\ell} - |U_1|\binom{|B|}{k-\ell-1} \ge \frac{\eps_1}{2}\binom{|B|}{k-\ell}$ choices of the extensions.
Note that we overall have obtained $(\frac{\eps_1}{4}\binom{|B|}{k-\ell})^3 \ge (\eps_1/4)^3\binom{|B|}{3(k-\ell)}$ distinct $\ell$-paths of length four, starting from the edge $e_1$.
Since $3k-3\ell\ge \ell$, this yields at least $(\eps_1/4)^3\binom{|B|}{\ell}$ distinct ordered $\ell$-sets as the (new) $\ell$-ends of the path. 
By Fact~\ref{fact:no bad}, there are at most 
$2^{\ell}\eps_2\ell!
\binom{|B|}{\ell} < (\eps_1/4)^3\binom{|B|}{\ell}$
ordered $\ell$-sets in $B$ that contain at least one $\eps_1$-bad subset, we can fix a choice of linkable $\ell$-end from the choices above.

Note that we can treat ${\bf L}_1$ similarly and thus obtain an $\ell$-path of length at most seven which contains $e_1$ (as an edge), has two linkable $\ell$-ends in $B'$, and
avoids \textit{any} given vertex set $U_1\subseteq B'$ of size at most $skq$.
Therefore the proof of the claim is completed.
\end{proof}

Now we turn to the second case, that is, for $k/2<\ell < k$, and we assume $\delta_{k-1}(\cH)\geq \frac{n}{s(k-\ell)}+\frac{k^2}2$.
In this case we build the $\ell$-paths in a different way: we shall show that there exist special stars, that is, a set of $\Omega(n^{k-2})$ edges that contain a common vertex $v$, with the property that all bad subsets of the edges contain $v$.
Thus, we can build $\ell$-paths by the Tur\'an number of tight paths, and the key property is that such (short) $\ell$-path contains $v$ right in the \textit{middle} of the path (see Figure~\ref{fig:2}), so that the $\ell$-ends are automatically linkable.

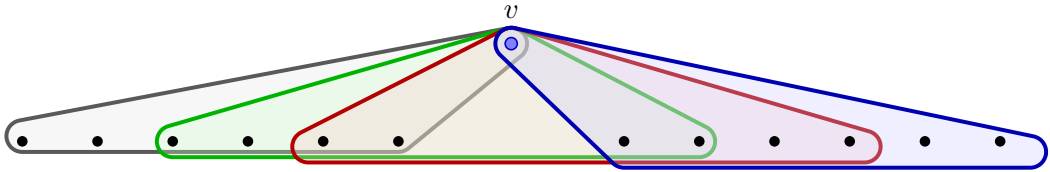
\begin{figure}[ht]
    \centering
\begin{tikzpicture}

	\coordinate (v01) at (1,0);
	\coordinate (v02) at (2,0);
	\coordinate (v010) at (1,0.07);
	\coordinate (v020) at (2,0.07);
	\coordinate (v001) at (3,0.07);
	\coordinate (v002) at (4,0.07);
	\coordinate (v003) at (5,0.07);
	\coordinate (v004) at (6,0.07);
	
	\coordinate (v1) at (3,0);
	\coordinate (v2) at (4,0);
	\coordinate (v3) at (5,0);
	\coordinate (v4) at (6,0);

	\coordinate (v5) at (7.5,1.3);

	\coordinate (v6) at (9,0);
	\coordinate (v7) at (10,0);
	\coordinate (v8) at (11,0);
	\coordinate (v9) at (12,0);
	
	\coordinate (v32) at (4.8,-0.07);
	\coordinate (v42) at (6,-0.07);
	\coordinate (v62) at (9,-0.07);
	\coordinate (v72) at (10,-0.07);
	\coordinate (v82) at (11,-0.07);
	\coordinate (v92) at (12.2,-0.07);
	
	\coordinate (v10) at (13,0);
	\coordinate (v11) at (14,0);
	
	\coordinate (v63) at (9,-0.14);
	\coordinate (v73) at (10,-0.14);
	\coordinate (v83) at (11,-0.14);
	\coordinate (v93) at (12.2,-0.14);
	\coordinate (v103) at (13,-0.14);
	\coordinate (v113) at (14.4,-0.14);

	\begin{pgfonlayer}{front}
		
		\foreach \i in {v01, v02, v1, v2, v3, v4, v6, v7, v8, v9, v10, v11}
			\fill  (\i) circle (2pt);
			
		\foreach \i in {v5}{
			\draw[blue!75!black, very thick]  (\i) circle (2pt);
			\fill[blue!50!white]  (\i) circle (2pt);}
		
		\node[above=0.2cm] at (v5) {$v$};
	\end{pgfonlayer}
	
\qedgeSeven[offset=6pt, line width=1.5pt,  fill=gray!20!white, fill opacity=0.3, draw=gray!70!black]{(v004)}{(v003)}{(v002)}{(v001)}{(v020)}{(v010)}{(v5)};
	\qedgeSeven[offset=6pt, line width=1.5pt, fill=green!20!white, fill opacity=0.3, draw=green!70!black]{(v7)}{(v6)}{(v4)}{(v3)}{(v2)}{(v1)}{(v5)};
	\qedgeSeven[offset=6pt, line width=1.5pt, fill=red!20!white, fill opacity=0.3, draw=red!70!black]{(v92)}{(v82)}{(v72)}{(v62)}{(v42)}{(v32)}{(v5)};
	\qedgeSeven[offset=6pt, line width=1.5pt,  fill=blue!20!white, fill opacity=0.3, draw=blue!70!black]{(v113)}{(v103)}{(v93)}{(v83)}{(v73)}{(v63)}{(v5)};
\end{tikzpicture}
    \caption{$\ell$-path contains $v$ right in the middle of the path for $k=7$, $\ell=5$}
    \label{fig:2}
\end{figure}

\begin{lemma}
[Disjoint paths in $B'$ for $k/2<\ell < k$]
\label{lem:sq}
Suppose $k\ge 3,k/2<\ell < k$ such that $(k-\ell)\nmid k$ and $\delta_{k-1}(\cH)\geq \frac{n}{s(k-\ell)}+\frac{k^2}{2}$.
    Let $|A\cap B'|=q>0$. 
    Then there exist $sq$ vertex-disjoint $\ell$-paths 
    in $B'$, each of which has length at most $s$ and contains two linkable $\ell$-ends.
\end{lemma}

\begin{proof}
Since $A\cap B'\neq \emptyset$, applying Lemma~\ref{lem1}, we have $B\subseteq B'$. 
Since $|B'|=|B|+|B'\setminus B|=|B|+|A\cap B'|=\left\lfloor \left(1-\frac{1}{s(k - \ell)}\right) n \right\rfloor+q$, by~\eqref{aa1}, it follows that $\delta_{k-1}(\mathcal H[B'])\geq q+k^2/2$. 
Let $\cB$ be the collection of all $(k-1)$-sets in $B$
that contain no $\eps_1$-bad subset.
Denote by $\cH^*$ the $k$-graph with vertex set $B'$ 
and edge set $E(\cH^*)=\{e: e=\{v\}\cup K, v\in B',K\in \cB\}$. 
By Fact~\ref{fact:no bad} and $\delta_{k-1}(\mathcal H[B'])\geq q+k^2/2$, we derive a lower bound for $e(\cH^*)$, that is,  
\begin{equation*}\label{eqn: lower bound}
e(\cH^*)> \frac{1}{k} (1-2^{k-1}\eps_2)\binom{|B|}{k-1}(q+k^2/2).
\end{equation*}

Now assume that we have found $i< sq$ $\ell$-paths of length at most $s$ with two linkable $\ell$-ends, and we shall show that we can find yet another such path, vertex disjoint from the existing ones.
Denote the vertex set of the union of these $i$ paths by $V'$. 
Thus, we have $|V'|\le i(\ell+s(k-\ell))< 2ksq$.
By the maximum degree condition~\eqref{b1},
the vertices in $V'$ are  incident to at most $2ksq\cdot 2\eps_1\binom{|B|}{k-1}$
edges.
Note that it suffices to find an $\ell$-path of length $s$ in $\cH^*_1:=\cH^*[B'\setminus V']$ with linkable ends. 
First note that $\cB_1:=\cB[B'\setminus V']$ satisfies that $|\cB_1|\ge (1-2^{k-1}\eps_2)\binom{|B|}{k-1} - |V'|\binom{|B|}{k-2}\ge (1-2^{k}sk\eps_2)\binom{|B|}{k-1}$, as $|V'|\le 2ksq$ and $q\le \eps_2|B|$.
For every $v\in B'$, let $E_{\cH_1^*}(v)$ be the set of edges $e$ of $\cH_1^*$ such that $e\setminus\{v\}\in \cB$.
First we assume that there exist $u,v\in B'$ such that $E_{\cH_1^*}(v)\cap E_{\cH_1^*}(u)\neq \emptyset$.
In this case, let $e\in E_{\cH_1^*}(v)\cap E_{\cH_1^*}(u)$ and observe that all $\eps_1$-bad subsets of $e$ must contain both $u$ and $v$.
Therefore, if we order the vertices of $e$ as $(u,\dots, v)$, then this ordered $k$-set forms an $\ell$-path of length one with both $\ell$-ends linkable.
Take $e$ to be the $(i+1)$st path and we are done.
Thus, we may assume that for every $u, v\in B'$, we have $E_{\cH_1^*}(v)\cap E_{\cH_1^*}(u)= \emptyset$.
This yields that $e(\cH_1^*) + k e(\cH^*\setminus \cH_1^*) \ge (1-2^{k}sk\eps_2)\binom{|B|}{k-1}(q+k^2/2)$,
which in turn gives that an improved bound as
\[
e(\cH_1^*)> (1-2^{k}sk\eps_2)\binom{|B|}{k-1}(q+k^2/2) - 4sk^2q\eps_1\binom{|B|}{k-1} \ge (1-8sk^2\eps_1)\binom{|B|}{k-1}(q+k^2/2).
\]
Therefore, by the pigeonhole principle,  there exists $u\in B'\setminus V'$ such that
\begin{equation}\label{eqn:edge number1}
e_{\cH_1^*}(u) \ge \frac{e(\cH_1^*)}{|B'\setminus V'|} \ge (1-8sk^2\eps_1)\frac{q+k^2/2}{|B'|}\binom{|B|}{k-1}.
\end{equation}
Since $B\subseteq B'$, by Lemma~\ref{11}, we have $|B'|=|B|+|B'\setminus B|\le (1+\eps_2)|B|$ and thus~\eqref{eqn:edge number1} is at least 
 \[
 (1-8sk^2\eps_1)\binom{|B|-1}{k-2}\frac{|B|}{k-1}\cdot \frac{1+k^2/2}{(1+\eps_2)|B|}> \frac{k+1+1/k}{2}(1-8sk^2\eps_1)\binom{|B'|-1}{k-2}.
 \]
Fix such $u\in B'\setminus V'$ satisfying~\eqref{eqn:edge number1}, and we consider the link $(k-1)$-graph of $u$ in $\cH_1^*$. 
By Lemma \ref{turan number}, the upper bound of $\mathrm{ex}(|B'\setminus V'|-1, P^{k-1}_{k})$ is $\frac{k+1}{2} \binom{|B'|-1}{k-2}< \frac{k+1+1/k}{2}(1-8sk^2\eps_1)\binom{|B'|-1}{k-2}$
since $\eps_1\ll \frac{1}{sk^k}$.
Then, we can find a copy of $(k-1)$-uniform tight path $P^{k-1}_k$ in $E_{\cH^*_1}(u)$, which together with $u$, forms a copy of $k$-uniform tight path $Q$ of length $k$ (so of order $2k-1$).
Now we can take a subpath of $Q$ of order $\ell+s(k-\ell)$, so that on the path, there are at least $\ell$ vertices before $u$ and at least $\ell$ vertices after $u$.
Finally, this tight path contains a ($k$-uniform) $\ell$-path of length $s$, denoted by $P_{i+1}$, as a subgraph (and we keep the ordering of the vertices).
By construction, as all edges of $P_{i+1}$ are in $E_{\cH^*_1}(u)$, we know that all $\eps_1$-bad subsets of the edges contain $u$, which yields that both ends of $P_{i+1}$ are linkable (as they do not contain $u$).
Therefore the proof of the claim is completed.
\end{proof}

\section{Remarks}\label{sec:4}
In this paper we verified Conjecture~\ref{conj:co-degree} for $k/2<\ell<3k/4$.
Moreover, by our Theorem~\ref{thm:main result}, to completely resolve Conjecture~\ref{conj:co-degree}, it suffices to build a family of vertex-disjoint $\ell$-paths with linkable $\ell$-ends as stated in Theorem~\ref{thm:main result}.
For this goal we provided two different attempts, yielding exact results for $k/2<\ell< 3k/4$ and results up to an additive constant error $(k^2/2)$ for all $\ell$.

\bibliographystyle{abbrv}
\bibliography{refs}

\begin{thebibliography}{10}

\bibitem{AS2000}
N.~Alon and J.~H. Spencer.
\newblock {\em \text{The Probabilistic Method}, second edition}.
\newblock John Wiley and Sons, New York, 2000.

\bibitem{BJ2017}
J.~O. Bastos, G.~O. Mota, M.~Schacht, J.~Schnitzer, and F.~Schulenburg.
\newblock Loose {H}amiltonian cycles forced by large {${(k-2)}$}-degree --- approximate version.
\newblock {\em SIAM J. Discrete Math.}, 31(4):2328--2347, 2017.

\bibitem{BMSSS2018}
J.~O. Bastos, G.~O. Mota, M.~Schacht, J.~Schnitzer, and F.~Schulenburg.
\newblock Loose {H}amiltonian cycles forced by large {$(k-2)$}-degree --- sharp version.
\newblock {\em Contrib. Discrete Math.}, 13(2):88--100, 2018.

\bibitem{BGHS1978}
J.~C. Bermond, A.~Germa, M.~C. Heydemann, and D.~Sotteau.
\newblock Hypergraphes hamiltoniens.
\newblock In {\em Probl\`emes combinatoires et th\'eorie des graphes ({C}olloq. {I}nternat. {CNRS}, {U}niv. {O}rsay, {O}rsay, 1976)}, volume 260 of {\em Colloq. Internat. CNRS}, pages 39--43. CNRS, Paris, 1978.

\bibitem{BHS2013}
E.~Bu\ss, H.~H\`an, and M.~Schacht.
\newblock Minimum vertex degree conditions for loose {H}amilton cycles in {$3$}-uniform hypergraphs.
\newblock {\em J. Combin. Theory Ser. B}, 103(6):658--678, 2013.

\bibitem{CHWY2025}
Y.~Cheng, J.~Han, B.~Wang, G.~Wang, and D.~Yang.
\newblock Transversal {H}amilton cycle in hypergraph systems.
\newblock {\em SIAM J. Discrete Math.}, 39(1):55--74, 2025.

\bibitem{CM14}
A.~Czygrinow and T.~Molla.
\newblock Tight codegree condition for the existence of loose {H}amilton cycles in 3-graphs.
\newblock {\em SIAM J. Discrete Math.}, 28(1):67--76, 2014.

\bibitem{DIR1952}
G.~A. Dirac.
\newblock Some theorems on abstract graphs.
\newblock {\em Proc. London Math. Soc. (3)}, 2:69--81, 1952.

\bibitem{FJKMV20}
Z.~F\"uredi, T.~Jiang, A.~Kostochka, D.~Mubayi, and J.~Verstra\"ete.
\newblock Tight paths in convex geometric hypergraphs.
\newblock {\em Adv. Comb.}, pages Paper No. 1, 14, 2020.

\bibitem{GHZ19}
W.~Gao, J.~Han, and Y.~Zhao.
\newblock Codegree conditions for tiling complete {$k$}-partite {$k$}-graphs and loose cycles.
\newblock {\em Combin. Probab. Comput.}, 28(6):840--870, 2019.

\bibitem{GM2018}
F.~Garbe and R.~Mycroft.
\newblock {H}amilton cycles in hypergraphs below the {D}irac threshold.
\newblock {\em J. Combin. Theory Ser. B}, 133:153--210, 2018.

\bibitem{GPW12}
R.~Glebov, Y.~Person, and W.~Weps.
\newblock On extremal hypergraphs for {H}amiltonian cycles.
\newblock {\em European J. Combin.}, 33(4):544--555, 2012.

\bibitem{HHZ2022}
H.~H\`an, J.~Han, and Y.~Zhao.
\newblock Minimum degree thresholds for {H}amilton {$(k/2)$}-cycles in {$k$}-uniform hypergraphs.
\newblock {\em J. Combin. Theory Ser. B}, 153:105--148, 2022.

\bibitem{HS2010}
H.~H\`an and M.~Schacht.
\newblock Dirac-type results for loose {H}amilton cycles in uniform hypergraphs.
\newblock {\em J. Combin. Theory Ser. B}, 100(3):332--346, 2010.

\bibitem{HSW2025}
J.~Han, L.~Sun, and G.~Wang.
\newblock Minimum degree conditions for {H}amilton {$\ell$}-cycles in {$k$}-uniform hypergraphs.
\newblock {\em Electron. J. Combin.}, 32(1):Paper No. 1.35, 18, 2025.

\bibitem{HZ15}
J.~Han and Y.~Zhao.
\newblock Minimum codegree threshold for {H}amilton {$\ell$}-cycles in {$k$}-uniform hypergraphs.
\newblock {\em J. Combin. Theory Ser. A}, 132:194--223, 2015.

\bibitem{HZ2015}
J.~Han and Y.~Zhao.
\newblock Minimum vertex degree threshold for loose {H}amilton cycles in 3-uniform hypergraphs.
\newblock {\em Journal of Combinatorial Theory, Series B}, 114:70--96, 2015.

\bibitem{HZ2016}
J.~Han and Y.~Zhao.
\newblock Forbidding {H}amilton cycles in uniform hypergraphs.
\newblock {\em J. Combin. Theory Ser. A}, 143:107--115, 2016.

\bibitem{KK1999}
G.~Y. Katona and H.~A. Kierstead.
\newblock Hamiltonian chains in hypergraphs.
\newblock {\em J. Graph Theory}, 30(3):205--212, 1999.

\bibitem{PDRD2011}
P.~Keevash, D.~Kühn, R.~Mycroft, and D.~Osthus.
\newblock Loose {H}amilton cycles in hypergraphs.
\newblock {\em Discrete Mathematics}, 311(7):544--559, 2011.

\bibitem{KMO10}
D.~K\"uhn, R.~Mycroft, and D.~Osthus.
\newblock Hamilton {$\ell$}-cycles in uniform hypergraphs.
\newblock {\em J. Combin. Theory Ser. A}, 117(7):910--927, 2010.

\bibitem{KO2014}
D.~K\"uhn and D.~Osthus.
\newblock {H}amilton cycles in graphs and hypergraphs: an extremal perspective.
\newblock In {\em Proceedings of the {I}nternational {C}ongress of {M}athematicians---{S}eoul 2014. {V}ol. {IV}}, pages 381--406. Kyung Moon Sa, Seoul, 2014.

\bibitem{DD2006}
D.~Kühn and D.~Osthus.
\newblock Loose {H}amilton cycles in 3-uniform hypergraphs of high minimum degree.
\newblock {\em Journal of Combinatorial Theory, Series B}, 96(6):767--821, 2006.

\bibitem{LS2022}
R.~Lang and N.~Sanhueza-Matamala.
\newblock Minimum degree conditions for tight {H}amilton cycles.
\newblock {\em J. Lond. Math. Soc. (2)}, 105(4):2249--2323, 2022.

\bibitem{LS2007}
L.~Lu and L.~Sz\'ekely.
\newblock Using {L}ov\'asz local lemma in the space of random injections.
\newblock {\em Electron. J. Combin.}, 14(1), 2007.

\bibitem{MR2011}
K.~Markstr\"om and A.~Ruci\'nski.
\newblock Perfect matchings (and {H}amilton cycles) in hypergraphs with large degrees.
\newblock {\em European J. Combin.}, 32(5):677--687, 2011.

\bibitem{RRRSS2019}
C.~Reiher, V.~R\"odl, A.~Ruci\'nski, M.~Schacht, and E.~Szemer\'edi.
\newblock Minimum vertex degree condition for tight {H}amiltonian cycles in {$3$}-uniform hypergraphs.
\newblock {\em Proc. Lond. Math. Soc. (3)}, 119(2):409--439, 2019.

\bibitem{VA10}
V.~R\"odl and A.~Ruci\'nski.
\newblock Dirac-type questions for hypergraphs---a survey (or more problems for {E}ndre to solve).
\newblock In {\em An irregular mind}, volume~21 of {\em Bolyai Soc. Math. Stud.}, pages 561--590. J\'anos Bolyai Math. Soc., Budapest, 2010.

\bibitem{RRS2006}
V.~R\"odl, A.~Ruci\'nski, and E.~Szemer\'edi.
\newblock A {D}irac-type theorem for {$3$}-uniform hypergraphs.
\newblock {\em Combin. Probab. Comput.}, 15(1-2):229--251, 2006.

\bibitem{RRS2008}
V.~R\"odl, A.~Ruci\'nski, and E.~Szemer\'edi.
\newblock An approximate {D}irac-type theorem for {$k$}-uniform hypergraphs.
\newblock {\em Combinatorica}, 28(2):229--260, 2008.

\bibitem{RRS11}
V.~R\"odl, A.~Ruci\'nski, and E.~Szemer\'edi.
\newblock Dirac-type conditions for {H}amiltonian paths and cycles in {$3$}-uniform hypergraphs.
\newblock {\em Adv. Math.}, 227(3):1225--1299, 2011.

\bibitem{Szm1978}
E.~Szemer\'edi.
\newblock Regular partitions of graphs.
\newblock In {\em Probl\`emes combinatoires et th\'eorie des graphes ({C}olloq. {I}nternat. {CNRS}, {U}niv. {O}rsay, {O}rsay, 1976)}, volume 260 of {\em Colloq. Internat. CNRS}, pages 399--401. CNRS, Paris, 1978.

\end{thebibliography}

\end{document}